\documentclass{lms}

\usepackage{graphics}
\usepackage{latexsym,bm}
\usepackage{amsmath,amsfonts,amssymb}
\usepackage{amscd}

  \newcommand{\bb}{\mathfrak{B}}
\newcommand{\bc}{\underline{c}}  \newcommand{\bi}{\underline{i}}
\newcommand{\bj}{\underline{j}}  \newcommand\bk{\underline{b}}
  
  \newcommand{\bU}{\rm U}
\newcommand{\BS}{\mathfrak S}
\newcommand{\CM}{\mathcal M}    

\newcommand{\fs}{\mathfrak s}   
\newcommand{\fu}{\mathfrak u}
\newcommand{\ft}{\mathfrak t}

    \newcommand{\lam}{\lambda}
\newcommand{\mU}{\mathbf U}
\newcommand{\PP}{\mathcal P}
\newcommand{\Q}{\mathbb Q}

\newcommand{\T}{\mathcal T}

\newcommand{\ui}{\underline i}

\newcommand{\uk}{\underline k}
\newcommand{\uj}{\underline j}

\newcommand{\Z}{\mathbb Z}

\DeclareMathOperator{\ann}{ann} \DeclareMathOperator{\BD}{Bd}
\DeclareMathOperator{\ch}{char}  \DeclareMathOperator{\diag}{diag}
\DeclareMathOperator{\End}{End} \DeclareMathOperator{\Ext}{Ext}
 \DeclareMathOperator{\Hom}{Hom}
\DeclareMathOperator{\id}{id} \DeclareMathOperator{\Ind}{Ind}
 \DeclareMathOperator{\Ker}{Ker}
\DeclareMathOperator{\Res}{Res} \DeclareMathOperator{\RS}{RowStd}
\DeclareMathOperator{\SGN}{SGN} \DeclareMathOperator{\res}{res}
\DeclareMathOperator{\Std}{Std} \DeclareMathOperator{\Span}{Span}
\DeclareMathOperator{\fS}{S} \DeclareMathOperator{\fT}{T}

\numberwithin{equation}{section}

\newtheorem{thm}[equation]{Theorem}
\newtheorem{prop}[equation]{Proposition}
\newtheorem{lem}[equation]{Lemma}
\newtheorem{cor}[equation]{Corollary}

\newnumbered{dfn}[equation]{Definition}
\newnumbered{rem}[equation]{Remark}

\title[Schur--Weyl duality]{Schur--Weyl duality for orthogonal groups}
\author{Stephen Doty and Jun Hu}
\classno{20G05 (primary), 20C20 (secondary).}

\extraline{The second author was supported by the National Natural
Science Foundation of China (Project 10771014), the Program NCET and
the Scientific Research Foundation for the Returned Overseas Chinese
Scholars, State Education Ministry.}

\begin{document}
\maketitle

\begin{abstract}
We prove Schur--Weyl duality between the Brauer algebra $\bb_n(m)$ and
the orthogonal group $O_{m}(K)$ over an arbitrary infinite field $K$
of odd characteristic. If $m$ is even, we show that each connected
component of the orthogonal monoid is a normal variety; this implies
that the orthogonal Schur algebra associated to the identity component
is a generalized Schur algebra. As an application of the main result,
an explicit and characteristic-free description of the annihilator of
$n$-tensor space $V^{\otimes n}$ in the Brauer algebra $\bb_n(m)$
is also given.
\end{abstract}

\section{Introduction}

Let $m, n\in\mathbb{N}$.  Write $\lambda
\vdash n$ to mean that $\lambda = (\lambda_1, \lambda_2, \dots)$ is a
partition of $n$, and denote by $\ell(\lam)$ the largest integer $i$
such that $\lam_i\neq 0$.

Let $K$ be an infinite field and $V$ an $m$-dimensional $K$-vector
space. The natural left action of the general linear group $GL(V)$ on
$V^{\otimes n}$ commutes with the right permutation action of the
symmetric group $\BS_n$. Let $\varphi, \psi$ be the corresponding
natural representations $$
\varphi:(K\BS_n)^{\mathrm{op}}\rightarrow\End_{K}\bigl(V^{\otimes
  n}\bigr),\quad \psi:KGL(V)\rightarrow\End_{K}\bigl(V^{\otimes
  n}\bigr), $$ respectively. The well-known Schur--Weyl duality (see
\cite{CL}, \cite{DP}, \cite{KSX}, \cite{Sc}, \cite{W}) says that
\begin{enumerate}
\item[(a)]
  $\varphi\bigl((K\BS_n)^{\mathrm{op}}\bigr)=\End_{KGL(V)}\bigl(V^{\otimes
  n}\bigr)$, and if $m\geq n$ then $\varphi$ is injective, and hence
  an isomorphism onto $\End_{KGL(V)}\bigl(V^{\otimes n}\bigr)$,
\item[(b)] $\psi\bigl(KGL(V)\bigr)=\End_{K\BS_n}\bigl(V^{\otimes n}\bigr)$,
\item[(c)] if $\ch K=0$, then there is an irreducible
  $KGL(V)$-$K\BS_n$-bimodule decomposition $$ V^{\otimes
  n}=\bigoplus_{\substack{\lam=(\lam_1,\lam_2,\cdots)\vdash
    n\\ \ell(\lam)\leq m}} {\Delta}_{\lam}\otimes S^{\lam},$$ where
  ${\Delta}_{\lam}$ (resp., $S^{\lam}$) denotes the irreducible
  $KGL(V)$-module (resp., irreducible $K\BS_n$-module) associated to
  $\lam$.
\end{enumerate}

There are also Schur--Weyl dualities for symplectic groups and
orthogonal groups in the semisimple case, i.e.,\ when $K$ has
characteristic zero; see \cite{B}, \cite{B1} and \cite{B2}. In these
cases, the symmetric group will be replaced by certain specialized
Brauer algebras. We are mostly interested in the non-semisimple
case. In \cite{DDH}, Schur--Weyl duality between the Brauer algebra
$\bb_n(-2m)$ and the symplectic group ${\rm Sp}_{2m}(K)$ over an
arbitrary infinite field $K$ was proved. In \cite{Hu2}, the second
author gave an explicit and characteristic-free description of the
annihilator of $n$-tensor space $V^{\otimes n}$ in the Brauer
algebra $\bb_n(-2m)$.\smallskip

The aim of this work is to generalize these results to the
orthogonal case. We first recall the definition of orthogonal group
over an arbitrary infinite field $K$ with $\ch K\neq 2$. Let $V$ be
an $m$-dimensional $K$-vector space with a non-degenerate symmetric
bilinear form $(\,,)$. Then the orthogonal similitude group (resp.,
orthogonal group) relative to $(\,,)$ is
$$ GO(V):=\biggl\{g\in GL(V)\biggm|\begin{matrix}\text{$\exists\, 0\neq d\in
K$, such that $(gv,gw)=d(v,w)$}\\
\text{$\forall\,v,w\in V$}\end{matrix}\biggr\}
$$
$\Bigl($resp., $O(V):=\Bigl\{g\in GL(V)\Bigm|(gv,
gw)=(v,w),\,\,\forall\,\,v,w\in V\Bigr\}.\,\,\Bigr) $\smallskip

By restriction from $GL(V)$, we get natural left actions of
$GO(V)$ and $O(V)$ on $V^{\otimes n}$. Note that if $0\neq d\in K$
is such that $(gv,gw)=d(v,w)$ for any $v,w\in V$, then
$\bigl((\sqrt{d^{-1}}g)v,(\sqrt{d^{-1}}g)w\bigr)=(v,w)$ for any
$v,w\in V$. Therefore, if $K$ is large enough such that
$\sqrt{d}\in K$ for any $d\in K$, then $g\in GO(V)$ implies that
$(a\id_{V})g \in O(V)$ for some $0\neq a\in K$. In that case,
$$\begin{aligned}
\psi(g)&=\psi\bigl((a^{-1}\id_{V})(a\id_{V})g\bigr)=\psi\bigl(a^{-1}\id_{V}\bigr)\psi\bigl((a\id_{V})g\bigr)\\
&=\bigl(a^{-n}\id_{V^{\otimes n
}}\bigr)\psi\bigl((a\id_{V})g\bigr)=a^{-n}\psi\bigl((a\id_{V})g\bigr).\end{aligned}
$$
It follows that
$$ \psi\bigl(K O(V)\bigr)=\psi\bigl(K GO(V)\bigr)$$
provided $K$ is closed under square roots.\smallskip

We now recall the definition of Brauer algebra. Let $x$ be an
indeterminate over $\Z$. The Brauer algebra $\bb_n(x)$ over $\Z[x]$ is
a unital $\Z[x]$-algebra with generators
$s_1,\cdots,s_{n-1},e_1,\cdots,e_{n-1}$ and relations (see
\cite{E}):
$$\begin{matrix}s_i^2=1,\,\,e_i^2=xe_i,\,\,e_is_i=e_i=s_ie_i,
\quad\forall\,1\leq i\leq n-1,\\
s_is_j=s_js_i,\,\,s_ie_j=e_js_i,\,\,e_ie_j=e_je_i,\quad\forall\,1\leq
i<j-1\leq n-2,\\ s_is_{i+1}s_i=s_{i+1}s_is_{i+1},\,\,
e_ie_{i+1}e_i=e_i,\,\,
e_{i+1}e_ie_{i+1}=e_{i+1},\,\,\forall\,1\leq
i\leq n-2,\\
s_ie_{i+1}e_i=s_{i+1}e_i,\,\,e_{i+1}e_is_{i+1}=e_{i+1}s_i,\quad\forall\,1\leq
i\leq n-2.\end{matrix}
$$
$\bb_n(x)$ is a free $\Z[x]$-module with rank $(2n-1)\cdot
(2n-3)\cdots 3\cdot 1$. For any commutative $\Z[x]$-algebra $R$ with
$x$ specialized to $\delta\in R$, we define
$\bb_n(\delta)_{R}:=R\otimes_{\Z[x]}\bb_n(x)$. This algebra was
first introduced by Richard Brauer (see \cite{B}) in order to
describe how the $n$-tensor space $V^{\otimes{n}}$ decomposes into
irreducible modules over the orthogonal group $O(V)$ or the
symplectic group $Sp(V)$, where $V$ is an orthogonal or symplectic
vector space. In Brauer's original formulation, the algebra
$\bb_n(x)$ was defined as the complex linear space with basis the
set $\BD_n$ of all Brauer $n$-diagrams, graphs on $2n$ vertices and
$n$ edges with the property that every vertex is incident to
precisely one edge. The multiplication of two Brauer $n$-diagrams is
defined using natural concatenation of diagrams. For more details,
we refer the readers to \cite{GW} and \cite{Hu2}. Note that the
subalgebra of $\bb_n(x)$ generated by $s_1,s_2,\cdots,s_{n-1}$ is
isomorphic to the group algebra of the symmetric group $\BS_{n}$
over $\Z[x]$.\smallskip

The Brauer algebra has been studied in a number of references, e.g.,
\cite{B}, \cite{B1}, \cite{B2}, \cite{CDM}, \cite{DDH}, \cite{DP},
\cite{FG}, \cite{Ga1}, \cite{Ga2}, \cite{Hu2}, \cite{HW1},
\cite{HW2}, \cite{LP}, \cite{Wz}. To set up a Schur--Weyl duality
for orthogonal groups, we only need certain specialized Brauer
algebras which we now recall. Let
$\bb_n(m):=\Z\otimes_{\Z[x]}\bb_n(x)$, where $\Z$ is regarded as
$\Z[x]$-algebra by specifying $x$ to $m$. Let
$\bb_n(m)_K:=K\otimes_{\Z}\bb_n(m)$, where $K$ is regarded as
$\Z$-algebra in the natural way. Then there is a right action of the
specialized Brauer algebra $\bb_n(m)_K$ on the $n$-tensor space
$V^{\otimes n}$ which commutes with the natural left action of
$GO(V)$. We recall the definition of this action. Let $\delta_{i,j}$
denote the value of the usual Kronecker delta. For any integer $i$
with $1\leq i\leq m$, we set $i'=m+1-i$. We fix an ordered basis
$\bigl\{v_1,v_2,\cdots,v_{m}\bigr\}$ of $V$ such that
$$ (v_i, v_{j})=\delta_{i, j'},\quad\forall\,\,1\leq i, j\leq m.$$
The right action of $\bb_n(m)$ on $V^{\otimes n}$ is defined on
generators by
$$\begin{aligned} (v_{i_1}\otimes\cdots\otimes
v_{i_n})s_j&:=v_{i_1}\otimes\cdots\otimes v_{i_{j-1}}\otimes
v_{i_{j+1}}\otimes v_{i_{j}}\otimes v_{i_{j+2}}
\otimes\cdots\otimes v_{i_n},\\
(v_{i_1}\otimes\cdots\otimes v_{i_n})e_j&:=\delta_{i_{j},i'_{j+1}}
v_{i_1}\otimes\cdots\otimes v_{i_{j-1}}\otimes\biggl(
\sum_{k=1}^{m}v_{k}\otimes v_{k'}\biggr)\otimes v_{i_{j+2}}\\
& \qquad\qquad\otimes\cdots\otimes v_{i_n}.\end{aligned}
$$
That is, the action of $s_j$ is by place permutation and the action
of $e_j$ is by a composition of Weyl's ``contraction'' operator with
an ``expansion'' operator. Let $\varphi$ be the $K$-algebra
homomorphism $$ \varphi:\,\,(\bb_n(m))^{\rm
op}\rightarrow\End_{K}\bigl(V^{\otimes n}\bigr)
$$
induced by the above action.

\begin{lem} {\rm (\cite{B}, \cite{B1}, \cite{B2})} \label{B1}
1) The natural left action of $GO(V)$ on $V^{\otimes
n}$ commutes with the right action of $\bb_n(m)$. Moreover, if
$K=\mathbb{C}$, then
$$\begin{aligned}
\varphi\bigl(\bb_n(m)_{\mathbb{C}}^{\mathrm{op}}\bigr)&=\End_{\mathbb{C}
GO(V)}\bigl(V^{\otimes n}\bigr)=\End_{\mathbb{C}
O(V)}\bigl(V^{\otimes n}\bigr),\\
\psi\bigl(\mathbb{C}GO(V)\bigr)&=\psi\bigl(\mathbb{C}
O(V)\bigr)=\End_{\bb_n(m)_{\mathbb{C}}}\bigl(V^{\otimes
n}\bigr),\end{aligned}
$$

2)  if $K=\mathbb{C}$ and $m\geq n$ then $\varphi$ is injective,
and hence an isomorphism onto
$\End_{\mathbb{C}GO(V)}\bigl(V^{\otimes n}\bigr)$,

3) if $K=\mathbb{C}$, then there is an irreducible
$\mathbb{C}GO(V)$-$\bb_n(m)_{\mathbb{C}}$-bimodule decomposition
$$
V^{\otimes n}=\bigoplus_{f=0}^{[n/2]}\bigoplus_{\substack{\lam\vdash n-2f\\
\lam'_1+\lam'_2\leq m}}\Delta({\lam})\otimes D({\lam}),$$ where
$\Delta({\lam})$ (respectively, $D({\lam})$) denotes the
irreducible $\mathbb{C}GO(V)$-module (respectively, the
irreducible $\bb_n(m)$-module) corresponding to $\lam$, and
$\lam'=(\lam'_1,\lam'_2,$ $\cdots)$ denotes the conjugate
partition of $\lam$.
\end{lem}

The first main result in this work removes the restriction on $K$ in
part 1) and part 2) of the above theorem. We have

\begin{thm} \label{mainthm1}
For any infinite field $K$ of odd characteristic, we
have \begin{enumerate}
\item[(a)] $\psi\bigl(K GO(V)\bigr)=\End_{\bb_n(m)}\bigl(V^{\otimes
n}\bigr)$; \item[(b)] $
\varphi\bigl(\bb_n(m)\bigr)=\End_{KGO(V)}\bigl(V^{\otimes n}\bigr)
=\End_{KO(V)}\bigl(V^{\otimes n}\bigr),$ and if $m\geq n$, then
$\varphi$ is also injective, and hence an isomorphism onto $$\End_{K
GO(V)}\bigl(V^{\otimes n}\bigr).$$
\end{enumerate}
\end{thm}

We remark that the first statement in part b) first appeared in \cite{DP} based on a completely different approach.
The algebra $$S_K^o(m,n):=\End_{\bb_n(m)}\bigl(V^{\otimes n}\bigr)$$ is
called the orthogonal Schur algebra associated to $GO(V)$. Note that we use a different definition of orthogonal Schur algebra in
Section 2 by defining $S_K^o(m,n)$ to be the linear dual of a certain coalgebra; the two
definitions are reconciled in \eqref{Oeh1}.
Let $\mathcal{R}=\mathbb{Z}[1/2]$. Let $V_{\mathcal{R}}$ be the free
$\mathcal{R}$-module generated by $v_1,\cdots,v_m$. Let
$\bb_n(m)_{\mathcal{R}}$ be the Brauer algebra defined over
$\mathcal{R}$. We set
$$S_{\mathcal{R}}^o(m,n):=\End_{\bb_n(m)_{\mathcal{R}}}\bigl(V_{\mathcal{R}}^{\otimes
n}\bigr).$$ In the course of our proof of Theorem \ref{mainthm1}, we
show that the orthogonal Schur algebra is stable under base change,
and the dimensions of both the orthogonal Schur algebra and the
endomorphism algebra $\End_{K GO(V)}\bigl(V^{\otimes n}\bigr)$ are
independent of the infinite field $K$ as long as ${\rm char} K\neq
2$, see Corollary \ref{basechange} and Lemma \ref{key2}. We have
also the following result:

\begin{thm}\label{mainthm15}
Let $K$ be an infinite field of odd characteristic
and let $OM_m(\overline{K})$ be the orthogonal monoid defined over the
algebraic closure of $K$. Suppose $m$ is even. Then
$OM_m(\overline{K})$ has two connected components, say
$OM_m^+(\overline{K}), OM_m^-(\overline{K})$, where
$OM_m^+(\overline{K})$ is the component containing the identity. Both
components $OM_m^+(\overline{K}), OM_m^-(\overline{K})$ are normal
varieties, and $OM_m^+(\overline{K})$ is a reductive normal algebraic
monoid. In particular, the orthogonal Schur algebra $S_K^{o,+}(m,n)$
associated to the identity component of $GO(V)$ is always a
generalized Schur algebra in the sense of \cite{Do1}, \cite{Do2}.
\end{thm}
When $m$ is odd, the orthogonal Schur algebra $S_K^{o}(m,n)$ is in
general not a generalized Schur algebra.  Note that the recent
papers \cite{Liu,Liu2} also study Schur algebras related to orthogonal
groups.\smallskip

As a consequence of Schur--Weyl duality, we know that the
annihilator of the tensor space $V^{\otimes n}$ in the Brauer
algebra $\bb_n(m)$ is stable under base change as long as $\ch K\neq
2$. Our second main result in this paper gives a characteristic-free
description of this annihilator.

\begin{thm} \label{mainthm2}
Let $K$ be an infinite field of odd characteristic and consider the
partition of $n$ given by
$(m+1,1^{n-m-1}):=(m+1,\underbrace{1,\cdots,1}_{\text{$n-m-1$
    copies}})$.  We have that $$
\Ker\varphi=\CM_{K}^{(m+1,1^{n-m-1})},
$$
where $\CM_{K}^{(m+1,1^{n-m-1})}$ is the right $K[\BS_{2n}]$-module
associated to $(m+1,1^{n-m-1})$ as defined in the paragraphs below
Lemma \ref{cor24}. In particular, $\Ker\varphi$ has a Specht
filtration, regarded as $K[\BS_{2n}]$-module.
\end{thm}\noindent
We refer the reader to Sections 6 and 7 for the definition of
$\CM_{K}^{(m+1,1^{n-m-1})}$ and the action of $\BS_{2n}$ on it.
\smallskip

The paper is organized as follows. In Section 2 we prove the
surjectivity of $\psi$. The proof is based on Cliff's basis for the
coordinate algebra of orthogonal groups and a generalized
Faddeev--Reshetikhin--Takhtajan's construction. We show that if $m$
is even, then each connected component of the orthogonal monoid is a
normal variety. This implies that the orthogonal Schur algebra
associated to the identity component is a generalized Schur algebra.
In Section 3 we develop a tilting module theory for the orthogonal
group $O_m(\overline{K})$. The main result there is that the tensor
product of two tilting modules over $O_m(\overline{K})$ is again a
tilting module. As a result, we deduce that the dimension of the
endomorphism algebra of tensor space $V^{\otimes n}$ as a module
over $O_m({K})$ does not depend on ${K}$ (for ${\rm char}K\neq 2$).
Based on the results in Section 3, the surjectivity of $\varphi$ in
the case where $m\geq n$ is proved in Section 4 in the same manner
as \cite[Section 3]{DDH}. In Section 5, we prove the surjectivity of
$\varphi$ in the case where $m\leq n$ in a similar way as
\cite[Section 4]{DDH}. In Section 6, we study a permutation action
of the symmetric group $\BS_{2n}$ on the Brauer algebra. We
construct a new $\Z$-basis for the resulting right
$\BS_{2n}$-module, which yields an integral filtration of Brauer
algebra by right $\BS_{2n}$-modules. Using these results and the
Schur--Weyl duality we have proved, we give in Section 7 an explicit
and characteristic-free description of the annihilator of tensor
space $V^{\otimes n}$ in the Brauer algebra $\bb_n(m)$.
\bigskip\bigskip

\section{Orthogonal monoid and orthogonal Schur algebra}

Let $R$ be a noetherian integral domain such that $2\cdot 1_{R}$ is
invertible in $R$. Let $x_{i,j}, 1\leq i,j\leq m$, be $m^2$ commuting
indeterminates over $R$. Let $A_{R}(m)$ be the free commutative
$R$-algebra (i.e., polynomial algebra) in these $x_{i,j}, 1\leq
i,j\leq m$. Let $I_{R}$ be the ideal of $A_{R}(m)$ generated by
elements of the form
\begin{equation} \label{rel21} \left\{\begin{aligned}
&\sum_{k=1}^{m}x_{k,i}x_{k',j},\,\,\, 1\leq i\neq j'\leq m;\\
&\sum_{k=1}^{m}x_{i,k}x_{j,k'},\,\,\, 1\leq i\neq j'\leq m;\\
&\sum_{k=1}^{m}(x_{k,i}x_{k',i'}-x_{j,k}x_{j',k'}),\,\,\, 1\leq i,j
\leq m.
\end{aligned}\right.
\end{equation}
The $R$-algebra $A_{R}(m)/{I_{R}}$ will be denoted by
$A_{R}^{o}(m)$. Write $c_{i,j}$ for the canonical image
$x_{i,j}+I_{R}$ of $x_{i,j}$ in $A_{R}^{o}(m)$ ($1\leq i,j\leq m$).
Then in $A_{R}^{o}(m)$ we have the relations
\begin{equation}\label{rel22}
\left\{\begin{aligned}
&\sum_{k=1}^{m}c_{k,i}c_{k',j}=0,\,\,\, 1\leq i\neq j'\leq m;\\
&\sum_{k=1}^{m}c_{i,k}c_{j,k'}=0,\,\,\, 1\leq i\neq j'\leq m;\\
&\sum_{k=1}^{m}(c_{k,i}c_{k',i'}-c_{j,k}c_{j',k'})=0,\,\,\, 1\leq
i,j \leq m.
\end{aligned}\right.\end{equation}
Note that $A_{R}(m)$ is a graded algebra, $A_{R}(m)=\oplus_{n\geq 0}
A_{R}(m,n)$, where $A_{R}(m,n)$ is the subspace spanned by the
monomials of the form $x_{\bi,\bj}$ for $(\bi,\bj)\in I^2(m,n)$,
where
$$\begin{aligned}
&I(m,n):=\bigl\{\bi=(i_1,\cdots,i_n)\bigm|1\leq i_j\leq
m,\,\forall\,j\bigr\},\\
&I^2(m,n)=I(m,n)\times I(m,n),\quad x_{\bi,\bj}:=x_{i_1,j_1}\cdots
x_{i_n,j_n}.
\end{aligned}
$$
Since $I_{R}$ is a homogeneous ideal, $A_{R}^{o}(m)$ is graded too
and $$A_{R}^{o}(m)=\oplus_{n\geq 0}A_{R}^{o}(m,n),$$ where
$A_{R}^{o}(m,n)$ is the subspace spanned by the monomials of the
form $c_{\bi,\bj}$ for $(\bi,\bj)\in I^2(m,n)$, where $$
c_{\bi,\bj}:=c_{i_1,j_1}\cdots c_{i_n,j_n}.$$

{\it By convention, throughout this paper, we identify the symmetric
group $\BS_n$ with the set of maps acting on their arguments on the
right.}  In other words, if $\sigma \in \BS_n$ and $a\in
\{1,\dots,n\}$ we write $(a)\sigma$ for the value of $a$ under
$\sigma$. This convention carries the consequence that, when
considering the composition of two symmetric group elements, the
leftmost map is the first to act on its argument. For example, we
have $(1,2,3)(2,3)=(1,3)$ in the usual cycle notation.\smallskip

If one defines $$ \Delta(x_{\bi,\bj})=\sum_{\uk\in I(m,n)
}x_{\bi,\uk}\otimes
x_{\uk,\bj},\quad\varepsilon(x_{\bi,\bj})=\delta_{\bi,\bj},\,\,\forall\,\bi,\bj\in
I(m,n),\forall\,n,
$$
then the algebra $A_{R}(m)$ becomes a graded bialgebra, and each
$A_{R}(m,n)$ is a sub-coalgebra of $A_{R}(m)$. Its linear dual
$$S_{R}(m,n):=\Hom_{R}(A_{R}(m,n),R)$$ is the usual {\it Schur
algebra} over $R$ (see \cite{Gr}). Let
$$S_{R}^{o}(m,n):=\Hom_{R}(A_{R}^{o}(m,n),R)$$ the {\em orthogonal
  Schur algebra}.  It is clear that $A_{R}^{o}(m,n)$ is in fact a
quotient coalgebra of $A_{R}(m,n)$, hence $S_{R}^{o}(m,n)$ is a
subalgebra of $S_{R}(m,n)$.\smallskip

For any integers $i,j\in\{1,2,\cdots,m\}$, we let $E_{i,j}$
denote the corresponding matrix unit for
$\End_{R}(V_{R})$, where $V_R$ is a free $R$-module of rank $n$. We define
$$\beta:=\sum_{1\leq i,j\leq
m}E_{i,j}\otimes E_{j,i},\,\,\,\, \gamma:=\sum_{1\leq i,j\leq
m}E_{i,j}\otimes E_{i',j'}.
$$
For $i=1,2,\cdots,n-1$, we set $$ \beta_i:=\id_{V^{\otimes
i-1}}\otimes\beta\otimes\id_{V^{\otimes n-i-1}},\quad
\gamma_i:=\id_{V^{\otimes i-1}}\otimes\gamma\otimes\id_{V^{\otimes
n-i-1}}.
$$
By direct verification, it is easy to see that the map which sends
$s_i$ to $\beta_i$ and $e_i$ to $\gamma_i$ for each $1\leq i\leq
n-1$ extends to a representation of $\bb_n(m)$ on $V_{R}^{\otimes n}$
which is nothing but the representation we have defined above Lemma
1.1. By \cite[Theorem 3.3]{Oe1} and the discussion in \cite[Section
5]{Oe1}, we know that \begin{equation}\label{Oeh1}
\End_{\bb_n(m)_R}\Bigl(V_R^{\otimes n}\Bigr)\cong
S_{R}^{o}(m,n):=\Hom_{R}(A_{R}^{o}(m,n),R).
\end{equation}
Let $i$ be an integer with $1\leq i\leq m$. We define $$
{\det}_0=\sum_{k=1}^{m}c_{k,i}c_{k',i'}\in A_{R}^{o}(m,2).
$$
By the relations in (\ref{rel22}), we know that ${\det}_0$ does not
depend on the choice of $i$.  It is well known (and easy to check)
that ${\det}_0$ is a group-like element in the bialgebra $A^o_R(m)$.
Note that the relations in (\ref{rel21}) are equivalent to $$
C^{t}JC={\det}_0J,\quad CJC^t={\det}_0J,
$$
where $$ C:=(c_{i,j})_{n\times n},\quad J=\begin{pmatrix}0 &0
&\cdots &0 & 1\\
0 &0 &\cdots &1 & 0\\
\vdots &\vdots &\vdots &\vdots &\vdots\\
0& 1 & \cdots & 0 & 0\\
 1 & 0 & \cdots & 0 & 0
\end{pmatrix}_{n\times n}
$$
It follows that $\bigl({\det}_0\bigr)^{n}={\det}^2$, where
$$\det:=\sum_{k=1}^{m}(-1)^{\ell(w)}c_{1,w(1)}c_{2,w(2)}\cdots c_{n,w(n)}$$
denotes the usual determinant function.\smallskip

Let $\overline{K}$ be the algebraic closure of $K$. Let
$M_m(\overline{K})$ be the set of all $m\times m$ matrices over
$\overline{K}$. Then $M_m(\overline{K})$ is a linear algebraic
monoid over $\overline{K}$. We define the orthogonal monoid
$O\!M_m(\overline{K})$ as follows: $$
O\!M_m(\overline{K}):=\biggl\{A\in
M_m(\overline{K})\biggm|\begin{matrix}\text{there exists $d\in
\overline{K}$, such that}\\
\text{$A^t JA=AJA^t=dJ$}.
\end{matrix}\biggr\}.
$$
The coordinate algebra $\overline{K}[M_{m}]$ of $M_m(\overline{K})$
is isomorphic to
$A_{\overline{K}}(m):=A_{K}(m)\otimes_{K}\overline{K}$. The
coordinate algebra of the general linear group
$GL_{m}(\overline{K})$ is isomorphic to $\overline{K}[x_{i,j},
\det(x_{i,j})_{m\times m}^{-1}]_{1\leq i,j\leq m}$. The orthogonal
similitude group $GO_{m}(\overline{K})$ is defined as $$
GO_m(\overline{K}):=\biggl\{A\in
GL_m(\overline{K})\biggm|\begin{matrix}\text{there exists $d\in
\overline{K}^{\times}$, such that}\\
\text{$A^t JA=AJA^t=dJ$}.
\end{matrix}\biggr\}.
$$

The natural embedding $\iota_1: GO_{m}(\overline{K})\hookrightarrow
GL_{m}(\overline{K})$ induces a surjective map $\iota_1^{\#}:
\overline{K}[GL_{m}]\twoheadrightarrow \overline{K}[GO_{m}]$.
Similarly, the natural embedding $\iota_0:
O\!M_{m}(\overline{K})\hookrightarrow M_{m}(\overline{K})$ induces a
surjective map $\iota_0^{\#}: \overline{K}[M_{m}]\twoheadrightarrow
\overline{K}[O\!M_{m}]$. We use $\iota_2, \iota_3$ to denote the
natural inclusion from $GO_m(\overline{K})$ into
$O\!M_{m}(\overline{K})$ and the natural inclusion from
$GL_{m}(\overline{K})$ into $M_{m}(\overline{K})$ respectively. Note
that $GL_m(\overline{K})$ is a dense open subset of
$M_m(\overline{K})$, and by \cite[(6.6(e)), (7.6(g))]{Dt},
$GO_m(\overline{K})$ is a dense open subset of
$O\!M_m(\overline{K})$. Therefore, $\iota_2$ (resp., $\iota_3$)
induces an inclusion $\iota_2^{\#}$ from $\overline{K}[O\!M_{m}]$
into $\overline{K}[GO_m]$ (resp., an inclusion $\iota_3^{\#}$ from
$\overline{K}[M_{m}]$ into $\overline{K}[GL_{m}]$. We denote by
$\widetilde{A}_{\overline{K}}^{o}(m)$ (resp.,
$\widetilde{A}_{\overline{K}}^{o}(m,n)$) the image of
$A_{\overline{K}}(m)$ (resp., of $A_{\overline{K}}(m,n)$) in
$\overline{K}[GO_m]$. We define $\widetilde{A}_{K}^{o}(m)$ (resp.,
$\widetilde{A}_{K}^{o}(m,n)$) to be the image of $A_{K}(m)$ (resp.,
of $A_{K}(m,n)$) under the surjective map
$\overline{K}[GL_{m}]\twoheadrightarrow \overline{K}[GO_{m}]$.
\smallskip

\begin{lem} \label{key1} With the notations as above, the algebra
$\widetilde{A}_{\overline{K}}^{o}(m)$ is isomorphic to
the coordinate algebra of the orthogonal monoid
$O\!M_{m}(\overline{K})$. Moreover, the $\overline{K}$-dimension of
$\widetilde{A}_{\overline{K}}^{o}(m,n)$ does not depend on the
choice of the infinite field $K$ as long as $\ch K\neq 2$.
\end{lem}

\begin{proof} We have the following commutative diagram of maps: $$
\begin{CD}
GO_m(\overline{K})
@>{\iota_1}>>GL_m(\overline{K})\\
@V{\iota_2} VV @V{\iota_3}VV\\
O\!M_m(\overline{K})@>{\iota_0}>>M_m(\overline{K})
\end{CD}\,\, ,
$$
which induces the following commutative diagram:
$$
\begin{CD}
\overline{K}[M_m]
@>{\iota_0^{\#}}>>\overline{K}[O\!M_m]\\
@V{\iota_3^{\#}} VV @V{\iota_2^{\#}}VV\\
\overline{K}[GL_m]@>{\iota_1^{\#}}>>\overline{K}[GO_m]
\end{CD}.
$$
Since $\iota_0^{\#}$ is a surjection, while $\iota_2^{\#}$ is an
injection, the first conclusion of the lemma follows immediately
from the above commutative diagram.\smallskip

It remains to prove the second conclusion. Let
$\mathcal{R}:=\mathbb{Z}[1/2]$. In \cite[Section 8]{C}, Cliff proved
that for any field $K$ which is an $\mathcal{R}$-algebra, the
elements in the following set
\begin{equation}\label{span1}
\biggl\{\bigl({\det}_0\bigr)^k[S:T]\biggm|\begin{matrix}\text{$k\in\mathbb{Z}, 0\leq k\leq n/2$, $[S:T]$ is $O(m)$}\\
\text{standard of shape $\lambda$, $\lambda\vdash n-2k$}
\end{matrix}\biggr\}
\end{equation}
forms a $K$-linear spanning set of $A^o_K(m,n)$.\smallskip

By definition of $O\!M_{m}(K)$, it is easy to check that the
defining relations (\ref{rel21}) vanish on every matrix in
$O\!M_{m}(K)$. It follows that there is a natural epimorphism of
graded bialgebras from $A_{K}^{o}(m)$ onto
$\widetilde{A}_{K}^{o}(m)$. Therefore, the image in
$\widetilde{A}_{K}^{o}(m)$ of the elements in (\ref{span1}) for all
$n\geq 0$ also form a $K$-linear spanning set of
$\widetilde{A}_{K}^{o}(m)$. On the other hand, since the coordinate
algebra of $GO_m(\overline{K})$ is just the localization of
$OM_m(\overline{K})$ at ${\det}_0$, it follows that the image in
$\overline{K}[GO_m]$ of the elements in the following set
\begin{equation}\label{span2}
\biggl\{\bigl({\det}_0\bigr)^k[S:T]\biggm|\begin{matrix}\text{$k\in\mathbb{Z}$,
$[S:T]$ is $O(m)$ standard }\\
\text{of shape $\lambda$, $\lambda\vdash r \in\Z^{\geq
0}$}\end{matrix} \biggr\}
\end{equation}
form a $\overline{K}$-linear spanning set of $\overline{K}[GO_m]$.
If $K=\mathbb{C}$, Cliff proved that (in \cite[Corollary 6.2]{C})
the natural image of the elements in (\ref{span1}) for all $n\geq 0$
is actually a basis of $\mathbb{C}[OM_m]$, from which we deduce that
the natural image of the elements in (\ref{span2}) are linearly
independent in $\mathbb{C}[GO_m]$.
\smallskip

By \cite[Chapter II, \S5, (2.7)]{DG}, we can deduce that the algebra
$\overline{K}[GO_m]$ has a nice $\mathcal{R}$-form
$\mathcal{R}[GO_m]$ such that the natural map $$
\mathcal{R}[GO_m]\otimes_{\mathcal{R}}\overline{K}\rightarrow
\overline{K}[GO_m]
$$
is an isomorphism. It follows that the elements in (\ref{span2}) are
always linearly independent in $\overline{K}[GO_m]$. Since
$\overline{K}[OM_m]$ is a subset of $\overline{K}[GO_m]$, we
conclude that the image of the elements (\ref{span1}) in
$\widetilde{A}_{\overline{K}}^{o}(m)$ also form a $K$-basis of
$\overline{K}[OM_m]=\widetilde{A}_{\overline{K}}^{o}(m)$. In
particular,
$\dim_{K}{A}_{{K}}^{o}(m,n)=\dim_{\overline{K}}\widetilde{A}_{\overline{K}}^{o}(m,n)$
is independent of the choice of the field $K$ as long as $\ch K\neq
2$. This completes the proof of the lemma.
\end{proof}

By Lemma \ref{key1}, for each $0\leq n\in\Z$, the dimension of
$\widetilde{A}_{K}^{o}(m,n)$ is independent of the field $K$. By
\cite[(9.5)]{Dt}, $A_{\mathbb C}^{o}(m,n)\cong
\widetilde{A}_{\mathbb C}^{o}(m,n)$. By (\ref{span1}),
${A}_{K}^{o}(m,n)$ has a spanning set which has the same cardinality
as $\dim \widetilde{A}_{\mathbb C}^{o}(m,n)$. Since
${A}_{K}^{o}(m,n)$ maps surjectively onto
$\widetilde{A}_{K}^{o}(m,n)$, and by Lemma \ref{key1},
$$\dim\widetilde{A}_{K}^{o}(m,n)=\dim\widetilde{A}_{\mathbb{C}}^{o}(m,n).
$$
It follows that the elements in the spanning set (\ref{span1}) form
an integral basis of ${A}_{K}^{o}(m,n)$, and thus the surjection
from ${A}_{K}^{o}(m,n)$ to $\widetilde{A}_{K}^{o}(m,n)$ is always an
isomorphism. It follows that
$A_{K}^{o}(m,n)\cong\widetilde{A}_{K}^{o}(m,n)$ and
$A_{K}^{o}(m)\cong\widetilde{A}_{K}^{o}(m)$. In particular, we have
$S_{K}^{o}(m,n)\cong\widetilde{S}_{K}^{o}(m,n):=\Hom_{K}\bigl(\widetilde{A}_{K}^{o}(m,n),K\bigr)$. Applying
(\ref{Oeh1}), we get that

\begin{cor} \label{basechange} With the notation as above, we have that\smallskip

1) $A^o_R(m,n)$ is a free $R$-module of finite rank, and for any
commutative $\mathcal{R}$-algebra $K$, the natural map
$$
A^o_{\mathcal{R}}(m,n)\otimes_{\mathcal{R}}K\rightarrow A^o_{K}(m,n)
$$
is always an isomorphism.

2) $\End_{\bb_n(m)_R}\Bigl(V_R^{\otimes n}\Bigr)$ is a free
$R$-module of finite rank, and for any commutative
$\mathcal{R}$-algebra $K$, the natural map
$$
\End_{\bb_n(m)_{\mathcal{R}}}\Bigl(V_{\mathcal{R}}^{\otimes
n}\Bigr)\otimes_{\mathcal{R}}K\rightarrow
\End_{\bb_n(m)_K}\Bigl(V_K^{\otimes n}\Bigr)
$$
is always an isomorphism.
\end{cor}

 By \cite[(4.4)]{Dt}, $GO_m(K)$ admits a graded polynomial
 representation theory in the sense of \cite[(1.2)]{Dt}. Applying
 \cite[(3.2)]{Dt}, we deduce that the images of $KGO(V)$ and of
 $\widetilde{S}_{K}^{o}(m,n)$ in $\End\bigl(V_K^{\otimes n}\bigr)$ are
 the same. On the other hand, the natural isomorphisms $
 \widetilde{S}_{K}^{o}(m,n)\cong S_{K}^{o}(m,n)\cong
\End_{\bb_n(m)}\bigl(V^{\otimes n}\bigr) $ imply that the image of
$\widetilde{S}_{K}^{o}(m,n)$ in $\End\bigl(V^{\otimes n}\bigr)$ is
exactly $\End_{\bb_n(m)}\bigl(V^{\otimes n}\bigr)$. Therefore, we
deduce that
$$ \psi(KGO(V))=\End_{\bb_n(m)}\bigl(V^{\otimes n}\bigr). $$
This completes the proof of part a) in Theorem \ref{mainthm1}.
It also shows the isomorphism
\[
  S_K^o(m,n) \cong \End_{\bb_n(m)}\bigl(V^{\otimes n}\bigr);
\]
so we see that the orthogonal Schur algebra may be regarded as an
endomorphism algebra for the Brauer algebra.

\medskip

From now on until the end of this section, we consider only the case where $m=2l$.
In \cite[Section 8]{C}, Cliff proved that
$GO_{2l}(\overline{K})$ is isomorphic to $O_{2l}(\overline{K})\times
\overline{K}^{\times}$ as a variety. The isomorphism is given by
$$\begin{aligned}
\rho_1: \quad & GO_{2l}(\overline{K})\rightarrow
O_{2l}(\overline{K})\times
\overline{K}^{\times}\\
&\quad A\mapsto \biggl(A\xi_1({\det}_0A)^{-1}, {\det}_0A\biggr),
\end{aligned}
$$
where $\xi_1(t):=\diag\bigl(\underbrace{t,\cdots,t}_{\text{$l$
copies}},\underbrace{1,\cdots,1}_{\text{$l$ copies}}\bigr),
\forall\,t\in \overline{K}$. The inverse of $\rho_1$ is given by
$$\begin{aligned} \rho_1^{-1}:\quad &
O_{2l}(\overline{K})\times
\overline{K}^{\times}\rightarrow GO_{2l}(\overline{K})\\
&\quad (A, c)\mapsto A\xi_1(c).
\end{aligned}
$$
In this case, we have that $\overline{K}[GO_{2l}]\cong
\overline{K}[O_{2l}]\otimes \overline{K}[T,T^{-1}]$. It also follows that
$GO_{2l}(\overline{K})$ has two connected components in this
case.\smallskip

As a regular function on $GO_{2l}(\overline{K})$,
${\det}^2={\det}_0^{2l}$. Note that ${\det}/{\det}_0$ is also a
regular function on $GO_{2l}(\overline{K})$. It follows that the two
connected components of $GO_{2l}(\overline{K})$ must be
$$\begin{aligned}
GO_{2l}^{+}(\overline{K}):&=\bigl\{A\in GO_{2l}(\overline{K})\bigm|{\det}A=({\det}_0A)^l\bigr\},\\
GO_{2l}^{-}(\overline{K}):&=\bigl\{A\in
GO_{2l}(\overline{K})\bigm|{\det}A=-({\det}_0A)^l\bigr\}.
\end{aligned}
$$
Note that in this case, ${\det}J=-1, J^2=I_{2l\times 2l}$, $J\in
GO_{2l}^{-}(\overline{K})$, and we have
$GO_{2l}^{-}(\overline{K})=GO_{2l}^{+}(\overline{K})\cdot J=J\cdot
GO_{2l}^{+}(\overline{K})$. It is easy to check that
$GO_{2l}^{+}(\overline{K})$ is a (connected) reductive algebraic group.
Let $$
O_{2l}^{+}(\overline{K}):=SO_{2l}(\overline{K}),\,\,\,O_{2l}^{-}(\overline{K}):=J\cdot
SO_{2l}(\overline{K}).
$$
It is clear from the isomorphism $\rho_1$ that
\begin{equation} \label{GOiso}
GO_{2l}^{+}(\overline{K})\cong O_{2l}^{+}(\overline{K})\times
K^{\times},\,\,\,GO_{2l}^{-}(\overline{K})\cong
O_{2l}^{-}(\overline{K})\times K^{\times}.
\end{equation}
Let $$\begin{aligned} OM_{2l}^{+}(\overline{K}):&=\bigl\{A\in
  OM_{2l}(\overline{K})\bigm|{\det}A=({\det}_0A)^l\bigr\},\\ OM_{2l}^{-}(\overline{K}):&=\bigl\{A\in
  OM_{2l}(\overline{K})\bigm|{\det}A=-({\det}_0A)^l\bigr\}.
\end{aligned}
$$
Since $\overline{GO_{2l}^{+}(\overline{K})}\subseteq
OM_{2l}^{+}(\overline{K})$, $\overline{GO_{2l}^{-}(\overline{K})}\subseteq
OM_{2l}^{-}(\overline{K})$, and ${GO_{2l}(\overline{K})}$ is a dense open
set in $OM_{2l}(\overline{K})$, it follows that
$GO_{2l}^{+}(\overline{K})$ (resp., $GO_{2l}^{-}(\overline{K})$) is a
dense open subset in $OM_{2l}^{+}(\overline{K})$ (resp., in
$OM_{2l}^{-}(\overline{K})$). It follows that $OM_{2l}^{+}(\overline{K}),
OM_{2l}^{-}(\overline{K})$ are the only two connected components
$OM_{2l}(\overline{K})$, and $I_{2l\times 2l} \in
OM_{2l}^+(\overline{K})$.

\begin{thm} \label{keythm1} Let $x$ be an indeterminant over
$\overline{K}$. Then there is an embedding
$\overline{K}[OM_{2l}^{+}]\hookrightarrow \overline{K}[O_{2l}^{+}]\otimes
\overline{K}[x]$, and we have the following commutative diagram
$$\begin{CD} \overline{K}[O\!M_{2l}^{+}]
@>{}>>\overline{K}[O_{2l}^{+}]\otimes \overline{K}[x]\\
@V{\iota_2^{\#}} VV @V{\id\otimes\widetilde{\iota}}VV\\
\overline{K}[GO_{2l}^{+}]@>{\sim}>>\overline{K}[O_{2l}^{+}]\otimes
\overline{K}[x,x^{-1}]
\end{CD}\,\, ,
$$
where the top horizontal map is the given embedding,
$\widetilde{\iota}$ is the natural embedding
$\overline{K}[x]\hookrightarrow\overline{K}[x,x^{-1}]$. The same is
true if we replace ``$+$'' by ``$-$''.
\end{thm}

\begin{proof} Let $f\in \overline{K}[O_{2l}^{+}]\otimes
\overline{K}[x,x^{-1}]$. We can write $$
f=\sum_{i\in\mathbb{Z}}f_i\otimes x^i,
$$ where $f_i\in\overline{K}[O_{2l}^{+}]$ for each $i$ and ${\rm
  supp}f:=\{i\in\mathbb{Z}|f_i\neq 0\}$ is a finite set. From
\eqref{GOiso} we have an isomorphism $\overline{K}[GO_{2l}^{+}]\cong
\overline{K}[O_{2l}^{+}]\otimes \overline{K}[x,x^{-1}]$. One can show that the
function $x^i$ on $GO_{2l}^{+}$ is given by ${\det}_0^i$. We identify
$\overline{K}[O\!M_{2l}^{+}]$ with its image in $\overline{K}[GO_{2l}^{+}]$,
and $\overline{K}[GO_{2l}^{+}]$ with $\overline{K}[O_{2l}^{+}]\otimes
\overline{K}[x,x^{-1}]$. We regard $O\!M_{2l}^{+}(\overline{K})$ as a closed
subvariety of $\bigl(\overline{K}\bigr)^{4l^2}$. Suppose that
$f\in\overline{K}[O\!M_{2l}^{+}]$. This means that $f$ can be extended to
a regular function $\widetilde{f}$ on $O\!M_{2l}^{+}(\overline{K})$.  For
each element $A\in O\!M_{2l}^{+}(\overline{K})\setminus
GO_{2l}^{+}(\overline{K})$, we know (by definition) that ${\det}_0A=0$.
Since $\widetilde{f}\in\overline{K}[O\!M_{2l}^{+}]$, there must exist
an open neighborhood $V_{A}$ of $A$ and two polynomials $g_A,
h_A\in\overline{K}[x_{1,1},x_{1,2},\cdots,x_{2l,2l}]$, such that for any
$X\in V_A$, $h_A(X)\neq 0$ and $f(X)=g_A(X)/h_A(X)$. Note that the
open subsets $$ GO_{2l}^{+}(\overline{K}),\,\,\, V_{A}, \,\,A\in
O\!M_{2l}^{+}(\overline{K})\setminus GO_{2l}^{+}(\overline{K})$$ gives a
covering of $O\!M_{2l}^{+}(\overline{K})$. Since an affine variety is a
Noetherian topological space and $O\!M_{2l}^{+}(\overline{K})$ is an
infinite set, we can always find an element $A\in
O\!M_{2l}^{+}(\overline{K})\setminus GO_{2l}^{+}(\overline{K})$ such that
$$ V_{A}\bigcap \bigl(O\!M_{2l}^{+}(\overline{K})\setminus
GO_{2l}^{+}(\overline{K})\bigr)
$$
is an infinite set. We fix such an element $A$. We claim that
$f_i=0$ whenever $i<0$. Suppose this is not the case. Let $i_0<0$ be
the least integer such that $f_{i_0}\neq 0$. Then for any $X\in
V_A\cap GO_{2l}^{+}(\overline{K})$, $$
h_A(X)f_{i_0}(X)+h_A(X)\Bigl(\sum_{i_0<i\in\mathbb{Z}}f_i(X)
({\det}_0X)^{i-i_0}\Bigr)-({\det}_0X)^{-i_0}g_A(X)=0,
$$ Since $GO_{2l}^{+}(\overline{K})$ is dense in
$O\!M_{2l}^{+}(\overline{K})$, it follows that $V_A\cap
GO_{2l}^{+}(\overline{K})$ contains infinitely many points. This means we
have the following polynomial identity:

\begin{equation} \label{identity}
h_Af_{i_0}+h_A\Bigl(\sum_{i_0<i\in\mathbb{Z}}f_i
({\det}_0)^{i-i_0}\Bigr)-({\det}_0)^{-i_0}g_A=0.
\end{equation}
On the other hand, since $V_{A}\bigcap
\bigl(O\!M_{2l}^{+}(\overline{K})\setminus GO_{2l}^{+}(\overline{K})\bigr)$
is an infinite set, we can always find a point $B\in V_{A}\bigcap
\bigl(O\!M_{2l}^{+}(\overline{K})\setminus GO_{2l}^{+}(\overline{K})\bigr)$
such that $f_{i_0}(B)\neq 0$. Now we evaluate the polynomial identity
(\ref{identity}) at $B$ on both sides, we get a contradiction since
two of the terms on the left hand side of (\ref{identity}) are zero
(because $\det_0(B) = 0$) and the other is nonzero. The contradiction
proves our claim, and also completes the proof of the lemma.
\end{proof}

Note that the map $$ A\mapsto AJ,\,\,\,\forall\,A\in
OM_{2l}^-(\overline{K}),
$$
defines a variety isomorphism $OM_{2l}^-(\overline{K})\cong
OM_{2l}^+(\overline{K})$. The following corollary proves the statements
in Theorem \ref{mainthm15}.

\begin{cor} As a variety, $OM_{2l}^+(\overline{K})$ is normal,
and hence $OM_{2l}^{+}(\overline{K})$ is a (connected)
reductive normal algebraic monoid. In particular, in this case, $GO_{2l}^{+}$
admits a polynomial representation theory in the sense of \cite{Dt},
and
$$
\widetilde{S}_{{K}}^{o,+}(2l,n):=\Hom_{K}
\bigl(\widetilde{A}_{{K}}^{o,+}(2l,n),K\bigr)
$$
is a generalized Schur algebra in the sense of \cite{Do1} and
\cite{Do2}, where $\widetilde{A}_{{K}}^{o,+}(2l,n)$ denotes the image
of $A_K(2l,n)$ in the coordinate algebra $\overline{K}[OM_{2l}^+]$.
\end{cor}

\begin{proof}
Since $SO_{2l}(\overline{K})=O_{2l}^{+}(\overline{K})$ is an irreducible
smooth affine variety, it follows that $\overline{K}[SO_{2l}]$ is a
normal domain. By \cite[Exercise 4.18]{Ei}, $\overline{K}[SO_{2l}]\otimes
K[x]$ is a normal domain too.  Let $A:=\overline{K}[OM_{2l}^{+}]$,
$B:=\overline{K}[SO_{2l}]\otimes K[x]$. Then the fraction field of $A$ is
a subfield of the fraction field of $B$. Now let $z$ be an element in
the fraction field of $A$, such that $$
a_0+a_1z+\cdots+a_{k-1}z^{k-1}+z^k=0,
$$
for some $k\in\mathbb{N}\cup\{0\}$ and $a_0,a_1,\cdots,a_{k-1}\in
A$. Since $B$ is normal, it follows that $z\in B$. On the other
hand, note that ${\det}={\det}_0^{l}$ as a regular function on
$OM_{2l}^{+}(\overline{K})$, by the definition of the embedding
$A\hookrightarrow B$. From this it is easy to see that a polynomial
$f=\sum_{i\in\mathbb{Z}}f_i\otimes x^i\in B\setminus A$ if and only
if $$ f=g/{\det}_0^k,
$$ for some $k\in\Z^{\geq 0}$, $g\in
K[x_{1,1},x_{1,2},\cdots,x_{2l,2l}]$ satisfying ${\det}_0\nmid g$. It
follows that $B\setminus A$ is closed under multiplication. Applying
\cite[Chapter 5, Exercise 7]{AM}, we deduce that $x\in A$. This proves
that $A$ is normal.  Hence $OM_{2l}^{+}(\overline{K})$ is normal. Since
its group of units $GO_{2l}^{+}(\overline{K})$ is a reductive group, it
follows that $OM_{2l}^{+}(\overline{K})$ is a reductive normal algebraic
monoid. It is easy to check that $0\in OM_{2l}^{+}(\overline{K})$ and
$OM_{2l}^{+}(\overline{K})$ has one-dimensional center. Using
\cite[Theorem 4.4]{Dt2}, we deduce that $GO_{2l}^{+}$ admits a polynomial
representation theory in the sense of \cite{Dt}, and
$\widetilde{S}_{{K}}^{o,+}(2l,n) :=
\Hom_{K}\bigl(\widetilde{A}_{{K}}^{o,+}(2l,n),K\bigr)$ is a generalized
Schur algebra in the sense of \cite{Do1} and \cite{Do2}.
\end{proof}

Finally, we remark that the same argument can be used to show that the
symplectic monoid (\cite{Dt}, \cite{Oe1}) $SpM_{2l}(\overline{K})$ is also a
connected reductive normal algebraic monoid.
\bigskip\bigskip

\section{Tilting modules over orthogonal groups}

The purpose of this section is to develop a tilting module theory
for orthogonal groups. Note that in the literature the theory of
tilting modules was well established for connected reductive
algebraic groups, and the existence of a tilting module theory for
orthogonal groups was only announced in \cite{AR} without full
details.\smallskip

Let $K$ be an infinite field of odd characteristic, $\overline{K}$
be its algebraic closure. By restriction, $V$ becomes a module over
the special orthogonal group $SO_m(\overline{K})$. In this case,
$V\cong
L(\varepsilon_1)=\Delta(\varepsilon_1)=\nabla(\varepsilon_1)$ is a
tilting module over $SO_m(\overline{K})$. By the general theory of
tilting modules over semi-simple algebraic groups (cf. \cite[Chapter
E]{Ja}), we know that $V_{\overline{K}}^{\otimes n}$ is also a
tilting module over $SO_m(\overline{K})$, and the dimension of $$
\End_{SO_m(\overline{K})}\Bigl(V_{\overline{K}}^{\otimes n}\Bigr)
$$
does not depend on the choice of the field $K$. \smallskip

Let $\theta\in GL(V)$ which is defined on the basis
$\bigl\{v_{i}\bigr\}_{1\leq i\leq m}$ by
$$
\theta(v_i)=\begin{cases} v_{i'}, &\text{if $i=m/2$ or $i=m/2+1$;}\\
v_{i}, &\text{otherwise.}
\end{cases}, \,\,\,i=1,2,\cdots,m,$$
if $m$ is even; or
$$
\theta(v_i)=-v_i,\,\,\,i=1,2,\cdots,m,$$ if $m$ is odd. Note that
$\theta$ is an order $2$ element in $O_m(K)$, and $O_m(K)$ is
generated by $SO_m(K)$ and $\theta$. \smallskip

For the moment we assume that $m=2l$ is even, and $K=\overline{K}$.
Let $G:=O_m(K), H:=SO_m(K)$. We set
$$
T:=\bigl\{\diag(t_1,\cdots,t_l,t_l^{-1},\cdots,t_1^{-1})\bigm|t_1,\cdots,t_l\in
K^{\times}\bigr\}.
$$
Then $T$ is a closed subgroup of $H$. In fact, $T$ is a maximal
torus of $H$. Clearly, $\theta T\theta^{-1}=T$. Let $W:=N_H(T)/T$ be
the Weyl group of $H$. For each integer $i$ with $1\leq i\leq m$,
let $\varepsilon_i$ be the function which sends a diagonal matrix in
$GL_{m}$ to its $i$th element in the diagonal. We identify a weight
$\lambda\varepsilon_1+\cdots+\lambda_{l}\varepsilon_{l}\in X(T)$
with the sequence $\lambda=(\lambda_1,\cdots,\lambda_{l})$ of
integers. Let $s_0$ be the generator of the cyclic group
$\mathbb{Z}/2\mathbb{Z}$. There is a natural action of
$\mathbb{Z}/2\mathbb{Z}$ on $X(T)$ which is defined on generators
by: $$ s_0(\lambda)=(\lambda_1,\lambda_2,\cdots,-\lambda_l),\,\,
\,\forall\,\lambda=(\lambda_1,\lambda_2,\cdots,\lambda_l)\in X(T).
$$ For each $\lambda\in X(T)^{+}$ (the set of dominant weights), we
use $L(\lambda), \Delta(\lambda), \nabla(\lambda)$ to denote the
corresponding simple module, Weyl module and co-Weyl module over $H$
respectively. If $s_0(\lambda)=\lambda$, then we let $\theta$ act as
$\id$ (resp., as $-\id$) on the highest weight vector of
$\Delta(\lambda)$. It is well known that (see \cite[(5.2.2)]{GW}) this extends to a representation
of $G$ on $\Delta(\lambda)$. The resulting $G$-module will be denoted by
$\widetilde{\Delta}^{+}(\lambda)$ (resp., by
$\widetilde{\Delta}^{-}(\lambda)$). In this case, $$
\Ind_{H}^{G}\Delta({\lambda})\cong \widetilde{\Delta}^{+}(\lambda)\oplus
\widetilde{\Delta}^{-}(\lambda).
$$
If $s_0(\lambda)\neq\lambda$, then we set $$
\widetilde{\Delta}^0(\lambda):=\Ind_{H}^{G}\Delta({\lambda}),\quad \widetilde{\nabla}^0(\lambda):=\Ind_{H}^{G}\nabla({\lambda}).
$$

In a similar way, we can define
$\widetilde{L}^{+}(\lambda),\widetilde{L}^{-}(\lambda)$ if
$s_0(\lambda)=\lambda$; and $\widetilde{L}^0(\lambda)$ if $s_0(\lambda)\ne \lambda$.
Using the fact that $\Ind_{H}^{G}$ is an exact functor and $\theta$ permutes the set of $H$-submodules of any $G$-module $M$, we deduce easily the next lemma.

\begin{lem} With the notations as above, the set $$
\Bigl\{\widetilde{L}^{+}(\lambda),\widetilde{L}^{-}(\lambda),\widetilde{L}^0(\mu)\Bigm|
\lambda,\mu\in X(T)^{+}, s_0(\lambda)=\lambda, s_0(\mu)\neq\mu
\Bigr\}
$$
forms a complete set of pairwise non-isomorphic simple $G$-modules.
\end{lem}

If $s_0(\lam)=\lam$, then we define $\widetilde{\nabla}^{+}(\lambda), \widetilde{\nabla}^{-}(\lambda)$ to be the duals of $\widetilde{\Delta}^{+}(-w_0\lambda)$, $\widetilde{\Delta}^{-}(-w_0\lambda)$
(where $w_0$ is the longest element in the Weyl group of $H$) such that $\theta$ also acts as
$\id$ (resp., as $-\id$) on the highest weight vector of
$\widetilde{\nabla}^{+}(\lambda)$ (resp., of $\widetilde{\nabla}^{+}(\lambda)$).
In this case it is also easy to show that $$
\Ind_{H}^{G}\nabla({\lambda})\cong \widetilde{\nabla}^{+}(\lambda)\oplus
\widetilde{\nabla}^{-}(\lambda).
$$
We shall call
$\widetilde{\Delta}^{+}(\lambda), \widetilde{\Delta}^{-}(\lambda),
\widetilde{\Delta}^0(\lambda)$ the Weyl modules for $G$, and call
$\widetilde{\nabla}^{+}(\lambda)$, $\widetilde{\nabla}^{-}(\lambda),\widetilde{\nabla}^0(\lambda)$
the co-Weyl modules for $G$.

\begin{lem} \label{key21} Let $\lambda, \mu\in X(T)^{+}$,
$x,y\in\{+,-,0\}$, then $$
\Ext^{i}_{G}\Bigl(\widetilde{\Delta}^{x}(\lambda),\widetilde{\nabla}^{y}(\mu)\Bigr)=\begin{cases}
K, &\text{if $i=0$, $\lambda=\mu$ and $x=y$;}\\
0, &\text{otherwise.}
\end{cases}
$$
\end{lem}

\begin{proof} First, it is easy to see that if $x,y\in\{+,-\}$, then
\begin{equation}\label{key22}
\Hom_{G}\Bigl(\widetilde{\Delta}^{x}(\lambda),\widetilde{\nabla}^{y}(\mu)\Bigr)=\begin{cases}
K, &\text{if $\lambda=\mu$ and $x=y$;}\\
0, &\text{otherwise.}
\end{cases}
\end{equation}
Since $O_m$ is a flat group scheme (c.f. \cite[7.2]{C}) over
$\mathcal{R}$ and $SO_m$ is a normal subgroup scheme of $O_m$ (hence
$SO_m$ is exact in $O_m$), we can apply \cite[Part 1, Corollary
(4.6)]{Ja}. We divide the proof into two cases: \medskip

\noindent {\it Case 1.} $s_0(\mu)=\mu$, $y\in\{+,-\}$. By \cite[Part
1, Corollary (4.6)]{Ja}, we have $$\begin{aligned}
&\quad\,\Ext^{i}_{G}\Bigl(\widetilde{\Delta}^{x}(\lambda),\widetilde{\nabla}^{+}(\mu)\Bigr)\oplus
\Ext^{i}_{G}\Bigl(\widetilde{\Delta}^{x}(\lambda),\widetilde{\nabla}^{-}(\mu)\Bigr)\\
&=\Ext^{i}_{G}\Bigl(\widetilde{\Delta}^{x}(\lambda),\widetilde{\nabla}^{+}(\mu)\oplus\widetilde{\nabla}^{-}(\mu)\Bigr)
\\
&=\Ext^{i}_{G}\Bigl(\widetilde{\Delta}^{x}(\lambda),\Ind_{H}^{G}{\nabla}(\mu)\Bigr)\\
&\cong
\Ext^{i}_{H}\Bigl(\Res_{H}^{G}\bigl(\widetilde{\Delta}^{x}(\lambda)\bigr),{\nabla}(\mu)\Bigr)\\
&=\begin{cases} K, &\text{if $i=0$, $\lam=\mu$, $x\in\{+,-\}$}\\
0, &\text{otherwise.}
\end{cases}
\end{aligned}
$$
With (\ref{key22}), the above calculation shows that $$
\Ext^{i}_{G}\Bigl(\widetilde{\Delta}^{x}(\lambda),\widetilde{\nabla}^{y}(\mu)\Bigr)=\begin{cases}
K, &\text{if $i=0$, $\lambda=\mu$ and $x=y$;}\\
0, &\text{otherwise,}
\end{cases}
$$
as required.
\smallskip

\noindent {\it Case 2.} $s_0(\mu)\neq\mu$, $y=0$. By \cite[Part 1,
Corollary (4.6)]{Ja}, we have $$\begin{aligned}
\Ext^{i}_{G}\Bigl(\widetilde{\Delta}^{x}(\lambda),\widetilde{\nabla}^{0}(\mu)\Bigr)
&\cong\Ext^{i}_{G}\Bigl(\widetilde{\Delta}^{x}(\lambda),\Ind_{H}^{G}{\nabla}(\mu)\Bigr)\\
&\cong\Ext^{i}_{H}\Bigl(\Res_{H}^{G}\bigl(\widetilde{\Delta}^{x}(\lambda)\bigr),{\nabla}(\mu)\Bigr)\\
&=\begin{cases} K, &\text{if $i=0$, $\lam=\mu$, $x=0$}\\
0, &\text{otherwise,}
\end{cases}
\end{aligned}
$$
as required. This completes the proof of the lemma.
\end{proof}

\begin{rem} Recall that
$\mathcal{R}:=\mathbb{Z}[1/2]$. Let $H_{\mathcal{R}}$ be the
${\mathcal{R}}$-form of the special orthogonal group scheme $SO_m$.
We define (cf. \cite[Part I, (2.6)]{Ja})
$G_{\mathcal{R}}:=H_{\mathcal{R}}\rtimes\mathbb{Z}/2\mathbb{Z}$. Let
$\lambda\in X(T)^{+}$. By the representation theory of semi-simple
algebraic groups (cf. \cite[Lemma 11.5.3]{Do}, \cite[Part II, Chapter
  B]{Ja}), we know that both the Weyl module $\Delta(\lam)$ and the
co-Weyl module $\nabla(\lambda)$ have nice $\mathcal{R}$-forms. We
denote them by $\Delta_{\mathcal{R}}(\lam),
\nabla_{\mathcal{R}}(\lambda)$ respectively. Furthermore, for any
$\mu\in X(T)^{+}$,
$\nabla_{\mathcal{R}}(\lambda)\otimes\nabla_{\mathcal{R}}(\mu)$ has
a $\nabla$-filtration, i.e., a filtration of
$H_{\mathcal{R}}$-modules such that each successive quotient is
isomorphic to some $\nabla_{\mathcal{R}}(\nu)$ for some $\nu\in
X(T)^{+}$. The same is true for the $\Delta$-filtration of
$\Delta_{\mathcal{R}}(\lambda)\otimes\Delta_{\mathcal{R}}(\mu)$. As
a consequence, we can define the $R$-forms
$\widetilde{\Delta}_{\mathcal{R}}^{x}(\lambda),
\widetilde{\nabla}_{\mathcal{R}}^{x}(\lambda)$ $(x \in \{+,-,0\})$
in a similar way, and Lemma \ref{key21} remains true if we replace
everything by their ${\mathcal{R}}$-forms.
\end{rem}

For any finite dimensional $G$-module $M$, an ascending filtration
$0=M_0\subset M_1\subset\cdots\subset M$ of $G$-submodules is called
a $\widetilde{\Delta}$-filtration (resp.,
$\widetilde{\nabla}$-filtration) if each successive quotient is
isomorphic to some $\widetilde{\Delta}^{x}(\lambda)$ (resp., some
$\widetilde{\nabla}^{x}(\lambda)$), where $\lambda\in X(T)^{+},
x\in\{+,-,0\}$. A $G$-module is called a tilting module if it has
both $\widetilde{\Delta}$-filtration and
$\widetilde{\nabla}$-filtration.

\begin{lem} \label{key23} Let $\lam,\mu\in X(T)^{+}$, $x,y\in\{+,-,0\}$. The
$G$-module $\widetilde{\Delta}^{x}(\lambda)\otimes
\widetilde{\Delta}^{y}(\mu)$ has a $\Delta$-filtration and the
$G$-module $\widetilde{\nabla}^{x}(\lambda)\otimes
\widetilde{\nabla}^{y}(\mu)$ has a $\nabla$-filtration.
In particular, the tensor product of any two tilting modules over
$G$ is again a tilting module over $G$.
\end{lem}

\begin{proof} Let $M:=\widetilde{\Delta}_{\mathcal{R}}^{x}(\lambda)\otimes
\widetilde{\Delta}_{\mathcal{R}}^{y}(\mu)$. By the definition of
$\widetilde{\Delta}_{\mathcal{R}}^{x}(\lambda)$ and
$\widetilde{\Delta}_{\mathcal{R}}^{y}(\mu)$, it is easy to see that
$\Res_{H}^{G}M$ has $\Delta$-filtration (as
$H_{\mathcal{R}}$-module), say, $$ 0=M_0\subset M_1\subset
M_2\subset\cdots\subset M_s=M.
$$

We set $N:=M_1$. Then $N\cong{\Delta}_{\mathcal{R}}(\nu)$ for some
$\nu\in X(T)^+$. Let $\widetilde{N}:=N+\theta N$. Then
$\widetilde{N}$ is a $G_{\mathcal{R}}$-submodule of $M$. It is clear
that there is a surjective $G_{\mathcal{R}}$-homomorphism from
$\Ind_{H_{\mathcal{R}}}^{G_{\mathcal{R}}}\bigl({\Delta}_{\mathcal{R}}(\nu)\bigr)$
onto $\widetilde{N}$. We denote it by $\rho:
\Ind_{H_{\mathcal{R}}}^{G_{\mathcal{R}}}\bigl({\Delta}_{\mathcal{R}}(\nu)\bigr)\twoheadrightarrow
\widetilde{N}$.\smallskip

Note that both
$\Ind_{H_{\mathcal{R}}}^{G_{\mathcal{R}}}\bigl({\Delta}_{\mathcal{R}}(\nu)\bigr)$
and $\widetilde{N}$ are free $\mathcal{R}$-submodules of $M$. If
${\rm
rank}_{\mathcal{R}}\Ind_{H_{\mathcal{R}}}^{G_{\mathcal{R}}}\bigl({\Delta}_{\mathcal{R}}(\nu)\bigr)={\rm
rank}_{\mathcal{R}}\widetilde{N}$, then it is readily seen that
$\rho$ is an isomorphism. That is,
$\widetilde{N}\cong\Ind_{H_{\mathcal{R}}}^{G_{\mathcal{R}}}\bigl({\Delta}_{\mathcal{R}}(\nu)\bigr)$
as $G$-module. Note that $$
\Ind_{H_{\mathcal{R}}}^{G_{\mathcal{R}}}\bigl({\Delta}_{\mathcal{R}}(\nu)\bigr)=\widetilde{\Delta}_{\mathcal{R}}^0(\nu),
$$
if $s_0(\nu)\neq\nu$; or $$
\Ind_{H_{\mathcal{R}}}^{G_{\mathcal{R}}}\bigl({\Delta}_{\mathcal{R}}(\nu)\bigr)\cong
\widetilde{\Delta}_{\mathcal{R}}^+(\nu)\oplus
\widetilde{\Delta}_{\mathcal{R}}^-(\nu),
$$ if $s_0(\nu) = \nu$. Clearly
$\Res^{G_{\mathcal{R}}}_{H_{\mathcal{R}}}\bigl(\widetilde{N}\bigr)$
has a $\Delta$-filtration. Applying \cite[Part II, Lemma B.9,
Corollary 4.17)]{Ja}, both $M$ and $M/\widetilde{N}$ have
$\widetilde{\Delta}$-filtrations, so $M/\widetilde{N}$ also has a
$\widetilde{\Delta}$-filtration. Now it follows easily by induction
on $\dim M$ that $M$ has a $\widetilde{\Delta}$-filtration as
$G$-module, as required.\smallskip

Now we assume that
\begin{equation}
\label{assu} {\rm
rank}_{\mathcal{R}}\Ind_{H_{\mathcal{R}}}^{G_{\mathcal{R}}}\bigl({\Delta}_{\mathcal{R}}(\nu)\bigr)>{\rm
rank}_{\mathcal{R}}\widetilde{N}.\end{equation} Let $\phi_1, \phi_2$
be the following two maps: $$\begin{matrix} \phi_1:& \theta
N\otimes_{\mathcal{R}}\mathbb{C}\rightarrow\widetilde{N}\otimes_{\mathcal{R}}\mathbb{C}\quad
& \phi_2:&
N\otimes_{\mathcal{R}}\mathbb{C}\rightarrow\widetilde{N}\otimes_{\mathcal{R}}\mathbb{C}\\
&\,\,\theta x\otimes_{\mathcal{R}}c\mapsto
\theta x\otimes_{\mathcal{R}}c &\qquad
&\,\,x\otimes_{\mathcal{R}}c\mapsto x\otimes_{\mathcal{R}}c,
\end{matrix}
$$
where $x\in N, c\in\mathbb{C}$. Note that
$$
\theta{N}\otimes_{\mathcal{R}}\mathbb{C}\cong{\Delta}_{\mathbb{C}}(s_0(\nu)),\quad
{N}\otimes_{\mathcal{R}}\mathbb{C}\cong{\Delta}_{\mathbb{C}}(\nu)
$$
are two simple $SO_{m}(\mathbb{C})$-modules, and $$
\Res^{G_{\mathbb{C}}}_{H_{\mathbb{C}}}\bigl(\widetilde{N}\otimes_{\mathcal{R}}\mathbb{C}\bigr)={\rm
im}(\phi_1)+{\rm im}(\phi_2).
$$
Since $\theta^2=1$, it follows easily that $\phi_1\neq 0$ if and
only if $\phi_2\neq 0$. Therefore, it follows from our assumption
(\ref{assu}) that $s_0(\nu)=\nu$ and $$
\Res^{G_{\mathbb{C}}}_{H_{\mathbb{C}}}\bigl(\widetilde{N}\otimes_{\mathcal{R}}\mathbb{C}\bigr)=
{\Delta}_{\mathbb{C}}(\nu).
$$
Therefore, we deduce that as $G_{\mathbb{C}}$-module, either $$
\widetilde{N}\otimes_{\mathcal{R}}\mathbb{C}\cong
\widetilde{\Delta}^+_{\mathbb{C}}(\nu),
$$
or $$ \widetilde{N}\otimes_{\mathcal{R}}\mathbb{C}\cong
\widetilde{\Delta}^-_{\mathbb{C}}(\nu).
$$ In particular, $\theta$ always acts as a scalar on the highest
weight vector of $\widetilde{N}$, and the scalar is either $1$ or
$-1$. This implies that $N=\theta N$, and hence
$\widetilde{N}=N+\theta N=N$, and either $\widetilde{N}\cong
\widetilde{\Delta}^+_{\mathcal{R}}(\nu)$ or $\widetilde{N}\cong
\widetilde{\Delta}^-_{\mathcal{R}}(\nu)$, as required. Now using the
same argument as before, we can prove by induction that $M$ has a
$\widetilde{\Delta}$-filtration as $G$-module. This proves the first
statement of this lemma. Since every co-Weyl module is the dual of
some Weyl module, the statement for $\widetilde{\nabla}$-filtrations
follows immediately by taking duals. This completes the proof of the
lemma.
\end{proof}

\begin{lem} \label{key2} Let $m$ be an arbitrary natural number, and
$K$ be an arbitrary infinite field of odd characteristic. Then the
dimension of
$$ \End_{O_m({K})}\Bigl(V_{{K}}^{\otimes n}\Bigr)
$$
does not depend on the choice of the infinite field $K$.
\end{lem}

\begin{proof} First, we note that $$
\End_{O_m({K})}\Bigl(V_{{K}}^{\otimes n}\Bigr)\otimes
_{K}\overline{K}=\End_{O_m(\overline{K})}\Bigl(V_{\overline{K}}^{\otimes
n}\Bigr).
$$
Therefore, to prove the lemma, we can assume without less of
generality that $K=\overline{K}$.

If $m$ is odd, then $\theta$ acts as $-\id$ on the tensor space
$V^{\otimes n}$. In that case, it is clear that $$
\dim\End_{O_m({K})}\Bigl(V_{{K}}^{\otimes
n}\Bigr)=\dim\End_{SO_m({K})}\Bigl(V_{{K}}^{\otimes n}\Bigr).
$$
It follows that (cf. the discussion at the beginning of this
section) the dimension of
$$ \End_{O_m({K})}\Bigl(V_{{K}}^{\otimes n}\Bigr)
$$
does not depend on the choice of the field $K$ in this
case.\smallskip

Now we assume that $m$ is even. We apply Lemmas \ref{key21} and \ref{key23}. It follows by induction through the filtrations that the dimension of
$$ \End_{O_m({K})}\Bigl(V_{{K}}^{\otimes n}\Bigr)
$$
again does not depend on the choice of the field $K$ in this case.
This completes the proof of the lemma.
\end{proof}

\bigskip\bigskip

\section{Proof of part b) in Theorem \ref{mainthm1} in the case $m\geq n$}

The purpose of this section is to give a proof of part b) in Theorem
\ref{mainthm1} in the case where $m\geq n$. Throughout this section,
we assume that $m\geq n$.\smallskip

By Lemma \ref{B1} and Lemma \ref{key2}, we know that $$
\dim\End_{O_m({K})}\Bigl(V_{{K}}^{\otimes n}\Bigr)=\dim\bb_n(m).
$$
Therefore, in order to prove part b) in Theorem \ref{mainthm1} in
the case $m\geq n$, it suffices to show that $\varphi$ is injective
in that case. Without loss of generality, we can assume that
$K=\overline{K}$ is algebraically closed.

Our strategy to prove the injectivity of $\varphi$ is similar to
that used in \cite[Section 3]{DDH}. First, we make some conventions
on the left and right place permutation actions. Throughout the rest
of this paper, for any $\sigma,\tau\in\BS_n, a\in\{1,2,\cdots,n\}$,
we set
$$ (a)(\sigma\tau)=\bigl((a)\sigma\bigr)\tau,\quad
(\sigma\tau)(a)=\sigma\bigl(\tau(a)\bigr).
$$
In particular, we have $\sigma(a)=(a)\sigma^{-1}$. Therefore, for
any $\bi=(i_1,i_2,\cdots,i_n)\in I(m,n), w\in\BS_n$, we have
$$
\bi w=(i_1,i_2,\cdots,i_n)w=(i_{w(1)},i_{w(2)},\cdots,i_{w(n)}), $$
which gives the so-called right place permutation action:
$$
v_{\bi}w=(v_{i_1}\otimes\cdots\otimes v_{i_n})w=
v_{i_{w(1)}}\otimes\cdots\otimes v_{i_{w(n)}}=v_{\bi w}.
$$

We make a further reduction. Let $\widehat{V}$ be the same $K$-vector
space as $V$, endowed with a different non-degenerate
symmetric bilinear form $(\,,)_1$ as follows: $$ (v_i,
v_j)_1:=\delta_{i,j},\,\,\forall\,1\leq i,j\leq m.
$$
Then the orthogonal group relative to $(\,,)_1$ is defined to be
$$ {O}(\widehat{V}):=\Bigl\{g\in GL(V)\Bigm|(gv,
gw)_1=(v,w)_1,\,\,\forall\,\,v,w\in \widehat{V}\Bigr\}. $$ We fix an
element $c_0\in K$ such that $c_0^2=-1$. Then it is easy to see that
the following map $$ \phi:\, v_i\mapsto\begin{cases}
(v_i+v_{i'})/\sqrt{2},&\text{if $1\leq i\leq m/2$;}\\
c_0(v_{i}-v_{i'})/\sqrt{2},&\text{if $(m+1)/2<i\leq m$;}\\
v_i, &\text{if $i=(m+1)/2$.}
\end{cases}\,\,\forall\,1\leq i\leq m,
$$
extends to an isomorphism from the orthogonal space $\widehat{V}$
onto the orthogonal space $V$. We extend $\phi$ diagonally to an
isomorphism (still denoted by $\phi$) from $\widehat{V}^{\otimes n}$
onto $V^{\otimes n}$. Let $x\in O(\widehat{V}),
f\in\End_{K}\bigl(\widehat{V}^{\otimes n}\bigr)$. It is easy to see
that $x\in {O}(\widehat{V})$ if and only if $\phi x\phi^{-1}\in
O(V)$, and $f\in\End_{{O}(\widehat{V})}\bigl(\widehat{V}^{\otimes
n}\bigr)$ if and only if $\phi
f\phi^{-1}\in\End_{O(V)}\bigl(V^{\otimes n}\bigr)$. In other words,
the map $\widetilde{\phi}: f\mapsto \phi f\phi^{-1}$ defines an
isomorphism from the endomorphism algebra
$\End_{{O}(\widehat{V})}\bigl(\widehat{V}^{\otimes n}\bigr)$ onto
the endomorphism algebra $\End_{O(V)}\bigl(V^{\otimes n}\bigr)$.
Recall that we have a natural map $\varphi$ from $\bb_n(m)$ to
$\End_{O(V)}\bigl(V^{\otimes n}\bigr)$. Using the isomorphism
$\widetilde{\phi}$, we get a $K$-algebra homomorphism
$\widehat{\varphi}$ from $\bb_n(m)$ to
$\End_{{O}(\widehat{V})}\bigl(\widehat{V}^{\otimes n}\bigr)$ as follows: $$
X\mapsto \phi^{-1} \varphi(X) \phi,\,\,\forall\,X\in\bb_n(m).
$$ By
direct calculation, one can verify that for any
$\bi=(i_1,\cdots,i_n)\in I(m,n)$, $j\in\{1,2,\cdots,n-1\}$,
$$\begin{aligned} (v_{i_1}\otimes\cdots\otimes
v_{i_n})\widehat{\varphi}(s_j)&:=v_{i_1}\otimes\cdots\otimes
v_{i_{j-1}}\otimes v_{i_{j+1}}\otimes v_{i_{j}}\otimes v_{i_{j+2}}
\otimes\cdots\otimes v_{i_n},\\
(v_{i_1}\otimes\cdots\otimes
v_{i_n})\widehat{\varphi}(e_j)&:=\delta_{i_{j},i_{j+1}}
v_{i_1}\otimes\cdots\otimes v_{i_{j-1}}\otimes\biggl(
\sum_{k=1}^{m}v_{k}\otimes v_{k}\biggr)\otimes v_{i_{j+2}}\\
& \qquad\qquad\otimes\cdots\otimes v_{i_n}.\end{aligned}
$$
To prove $\varphi$ is injective, it suffices to prove that
$\widehat{\varphi}$ is injective. This will be done in the rest of
this section.\smallskip

In \cite{GL}, the Brauer algebra was shown to be cellular. Enyang
gave in \cite{E} an explicit combinatorial cellular basis for Brauer
algebra.  Enyang's basis is in some sense similar to the Murphy basis
for type $A$ Hecke algebra. It is indexed by certain bitableaux. In
the remaining part of this section we shall use Enyang's results from
\cite{E}. We shall only use his basis for the specialized Brauer
algebra $\bb_n(m)$. We first recall some notations and
notions.\smallskip

Let $n$ be a natural number. A bipartition of $n$ is a pair
$(\lam^{(1)},\lam^{(2)})$ of partitions of numbers $n_1$ and $n_2$
with $n_1+n_2=n$. The notions of Young diagram, bitableaux, etc.,
carry over easily. For example, if $\lam:=(\lam^{(1)},\lam^{(2)})$
is a bipartition of $n$, then a $\lam$-bitableau $\ft$ is defined to
be a bijective map from the Young diagram $[\lam]$ to the set
$\{1,2,\cdots,n\}$. Thus $\ft$ is a pair $(\ft^{(1)},\ft^{(2)})$ of
tableaux, where $\ft^{(1)}$ is a $\lam^{(1)}$-tableau and
$\ft^{(2)}$ is a $\lam^{(2)}$-tableau. A bitableau
$\ft=(\ft^{(1)},\ft^{(2)})$ is called row standard if the numbers
increase along rows in both $\ft^{(1)}$ and $\ft^{(2)}$. For each
integer $f$ with $0\leq f\leq [n/2]$, we set $\nu=\nu_{f}:=((2^f),
(n-2f))$, where $(2^f):=(\underbrace{2,2,\cdots,2}_{\text{$f$
copies}})$ and $(n-2f)$ are considered as partitions of $2f$ and
$n-2f$ respectively. So $\nu$ is a bipartition of $n$. Let
$\ft^{\nu}$ be the standard $\nu$-bitableau in which the numbers
$1,2,\cdots,n$ appear in order along successive rows of the first
component tableau, and then in order along successive rows of the
second component tableau. We define
$$
\mathfrak{D}_{f}:=\Biggl\{d\in\BS_n\Biggm|\begin{matrix}\text{$(\ft^{(1)},\ft^{(2)})=\ft^{\nu}d$
is row standard and the first }\\
\text{column of $\ft^{(1)}$ is an increasing sequence}\\
\text{when read from top to bottom}\\
\end{matrix}\Biggr\}.
$$
For each partition $\lambda$ of $n-2f$, we denote by $\Std(\lambda)$
the set of all the standard $\lambda$-tableaux with entries in
$\{2f+1,\cdots,n\}$. The initial tableau $\ft^{\lam}$ in this case
has the numbers $2f+1,\cdots,n$ in order along successive rows.

\begin{lem} \text{\rm (\cite{E})} \label{Enyangbasis}
For each $\lam\vdash n-2f$, $\fs,
\ft\in\Std(\lam)$, let $m_{\fs,\ft}$ be the corresponding Murphy
basis element (cf. \cite{Mu}) of the symmetric group algebra
$K\BS_{\{2f+1,\cdots,n\}}$. Then the set
$$ \biggl\{{d_1}^{\ast}e_1e_3\cdots
e_{2f-1}m_{\fs\ft}{d_2}\biggm|\begin{matrix}
\text{$0\leq f\leq [n/2]$, $\lam\vdash n-2f$, $\fs, \ft\in\Std(\lam)$,}\\
\text{$d_1, d_2\in\mathfrak{D}_{f}$}
\end{matrix}\biggr\}$$
is a cellular basis of the Brauer algebra $\bb_n(m)_{\mathbb{Z}}$.
\end{lem}

As a consequence, by combining Lemma \ref{Enyangbasis} and
\cite[(3.3)]{E}, we have

\begin{cor} \label{Enyangbasis2} With the above notations, the set
$$ \biggl\{{d_1}^{\ast}e_1e_3\cdots e_{2f-1}{\sigma}
{d_2}\biggm|\begin{matrix}\text{$0\leq f\leq [n/2]$,
$\sigma\in\BS_{\{2f+1,\cdots,n\}}$,}\\
\text{$d_1, d_2\in\mathfrak{D}_{f}$}
\end{matrix}\biggr\}$$
is a basis of the Brauer algebra $\bb_n(m)_{\mathbb{Z}}$.
\end{cor}

{\it From now on and until the end of this section, we shall regard
the tensor space $V^{\otimes n}$ as a module over $\bb_n(m)$ via
$\widehat{\varphi}$ (instead of $\varphi$).} To prove the
injectivity of $\widehat{\varphi}$, it suffices to show that the
annihilator $\ann_{\bb_n(m)}(V^{\otimes n})$ is $(0)$. Note that
$$ \ann_{\bb_n(m)}(V^{\otimes n})=\bigcap_{v\in V^{\otimes
n}}\ann_{\bb_n(m)}(v).
$$
Thus it is enough to calculate $\ann_{\bb_n(m)}(v)$ for some set of
chosen vectors $v\in V^{\otimes n}$ such that the intersection of
annihilators is $(0)$. We write $$
\ann(v)=\ann_{\bb_n(m)}(v):=\bigl\{x\in \bb_n(m) \bigm|vx=0\bigr\}.
$$

For each integer $f$ with $0\leq f\leq [n/2]$, we denote by
$B^{(f)}$ the two-sided ideal of $\bb_n(m)_{\mathbb{Z}}$ generated
by $e_1e_3\cdots e_{2f-1}$. Note that $B^{(f)}$ is spanned by all
the Brauer diagrams which contain at least $2f$ horizontal edges
($f$ edges in each of the top and the bottom rows in the diagrams).

For $\bi\in I(m,n)$, an ordered pair $(s,t)$ ($1\leq s<t\leq n$) is
called an {\it orthogonal pair} in $\bi$ if $i_s=i_t$. Two ordered
pairs $(s,t)$ and $(u,v)$ are called disjoint if
$\bigl\{s,t\bigr\}\cap\bigl\{u,v\bigr\}=\emptyset$. We define the
{\it orthogonal length} $\ell_o(v_{\bi})=\ell_o(\bi)$ to be the
maximal number of disjoint orthogonal pairs $(s,t)$ in $\bi$. Note
that if $f>\ell_o(v_{\bi})$, then clearly
$B^{(f)}\subseteq\ann(v_{\bi})$.

\begin{lem} \label{step1} $\ann_{\bb_n(m)}\bigl(V^{\otimes
n}\bigr)\subseteq B^{(1)}$.
\end{lem}

\begin{proof}  Let $x\in\ann_{\bb_n(m)}\bigl(V^{\otimes n}\bigr)$. Then we can write $x = y+z$ where $y \in K\mathfrak{S}_n$, $z\in B^{(1)}$, because the set of diagrams with at least one horizontal edge spans $B^{(1)}$ (or else see Corollary \ref{Enyangbasis2}).

Since $m\geq n$, the tensor $v:=v_1\otimes v_2\otimes\cdots\otimes
v_n$ is well-defined and $\ell_o(v)=0$. It follows that
$B^{(1)}\subseteq\ann(v)$. In particular, $vz=0$. Therefore $vx=0$
implies that $vy=0$.

On the other hand, since $v_1,\cdots,v_n$ are pairwise distinct and
$y\in K\BS_n$, it is clear that $vy=0$ implies that $y=0$.
Therefore, we conclude that $x=z\in B^{(1)}$, as required.
\end{proof}

Suppose that we have already shown $\ann_{\bb_n(m)}\bigl(V^{\otimes
n}\bigr)\subseteq B^{(f)}$ for some natural number $1\leq f\leq
[n/2]$. We want to show that $\ann_{\bb_n(m)}\bigl(V^{\otimes
n}\bigr)\subseteq B^{(f+1)}$. Let $$ \bc:=(1,1,2,2,\cdots,f,f).
$$
We define
$$ I_f:=\Bigl\{\bk=(b_{1},\cdots,b_{n-2f})\Bigm|\text{$2f+1\leq
b_{1}<\cdots<b_{n-2f}\leq m$}\Bigr\}. $$ It is clear that
$\ell_o(v_{\bc}\otimes v_{\bk})=f$ for all $\bk\in I_f$.
\smallskip

Following \cite{DDH}, we consider the subgroup $\Pi$ of
$\BS_{\{1,\cdots,2f\}}\leq\BS_n$ permuting the rows of
$\ft^{\nu^{(1)}}$ but keeping the entries in the rows fixed. $\Pi$
normalizes the stabilizer $\BS_{(2^f)}$ of $\ft^{\nu^{(1)}}$ in
$\BS_{2f}$. We set $\Psi:=\BS_{(2^f)}\rtimes\Pi$. By \cite[Lemma
3.7]{DDH}, we have
$$\BS_{2f}=\bigsqcup_{d\in\mathcal{D}_f}\Psi d,$$ where $\mathcal{D}_f:=\mathfrak{D}_{f}\bigcap\BS_{2f}$, and
``$\,\sqcup$" means a disjoint union. Let
$\mathcal{P}_f:=\{(i_1,\cdots,i_{2f})|1\leq i_1<\cdots<i_{2f}\leq
n\}$. For each $J\in\mathcal{P}_f$, we use $d_J$ to denote the
unique element in $\mathfrak{D}_{f}$ such that the first component
of $\ft^{\nu}d_J$ is the tableau obtained by inserting the integers
in $J$ in increasing order along successive rows in
$\ft^{\nu^{(1)}}$. Let $\widetilde{\mathcal{D}}_{(2f,n-2f)}$ be the
set of distinguished right coset representatives of
$\BS_{(2f,n-2f)}$ in $\BS_n$. Clearly
$d_{J}\in\widetilde{\mathcal{D}}_{(2f,n-2f)}$, and every element of
$\widetilde{\mathcal{D}}_{(2f,n-2f)}$ is of the form $d_J$ for some
$J\in\mathcal{P}_f$. By \cite[Lemma 3.8]{DDH},
$\mathfrak{D}_{f}=\bigsqcup_{J\in\mathcal{P}_f}\mathcal{D}_f
d_{J}$.\smallskip

The proof of the next two lemmas is similar to \cite[Lemma 3.9,
Lemma 3.10]{DDH} except some minor changes. For the reader's
convenience, we include the proof here.

\begin{lem} \label{key31} Let $\bk\in I_f$, $v=v_{\bc}\otimes v_{\bk}\in
V^{\otimes n}$. Let $1\neq d\in\BS_n$. If either
$d\not\in\BS_{(2f,n-2f)}$ or $d\in\mathcal{D}_f$, then
$d^{-1}ze_1e_3\cdots e_{2f-1}\in\ann(v)$ for any $z\in\Psi$.
\end{lem}

\begin{proof} If $d\not\in
\BS_{(2f,n-2f)}$. Then $d^{-1}$ is not an element of
$\BS_{(2f,n-2f)}$ too. In particular, there is some $j$, $2f+1\leq
j\leq n$, such that $1\leq jd^{-1}\leq 2f$, and hence the basis
vector $v_{b_j}$ with $2f+1\leq b_j\leq m$ appears at position
$jd^{-1}$ in $vd^{-1}$. However, $v_{b_j}$ occurs only once as a
factor in $vd^{-1}$ and hence for any $z\in\Psi$,
$0=vd^{-1}ze_{jd^{-1}-1}$ if $jd^{-1}$ is even,
$0=vd^{-1}ze_{jd^{-1}}$ if $jd^{-1}$ is odd. As the $e_i$'s in
$e_1e_3\cdots e_{2f-1}$ commute we have $vd^{-1}ze_1e_3\cdots
e_{2f-1}=0$ in this case. If
$d\in\mathcal{D}_f=\mathfrak{D}_{f}\cap\BS_{2f}$, then $d$ and hence
$d^{-1}$ as well is not contained in the subgroup $\Psi$ of
$\BS_{2f}$ defined above. Therefore there exists
$j\in\{1,3,\cdots,2f-1\}$ such that $jd^{-1}, (j+1)d^{-1}$ are not
in the same row of $\ft^{(2^{f})}d^{-1}$. Now we see similarly as
above that $ze_1e_3\cdots e_{2f-1}$ annihilates $vd^{-1}$ for any
$z\in\Psi$.
\end{proof}

\begin{lem} Let $S$ be the subset
$$\biggl\{d_1^{-1}e_1e_3\cdots e_{2f-1}\sigma d_2\biggm|\begin{matrix}
&\text{$d_1, d_2\in\mathfrak{D}_{f}$, $d_1\neq 1$,}\\
&\text{$\sigma\in\BS_{\{2f+1,\cdots,n\}}$}\\
\end{matrix}\biggr\}$$ of the basis (\ref{Enyangbasis2}) of
$\bb_n(m)$, and let $U$ be the subspace spanned by $S$. Then $$
B^{(f)}\cap\Bigl(\bigcap_{\bk\in I_{f}}\ann(v_{\bc}\otimes
v_{\bk})\Bigr)=B^{(f+1)}\oplus U.
$$
\end{lem}

\begin{proof} Since $\ell_o(v_{\bc}\otimes v_{\bk})=f$, it follows that
$B^{(f+1)}\subseteq\ann(v_{\bc}\otimes v_{\bk})$. This, together
with Lemma \ref{key31}, shows that the right-hand side is contained
in the left-hand side.

Now let $x\in B^{(f)}\cap\bigl(\cap_{\bk\in
I_{f}}\ann(v_{\bc}\otimes v_{\bk})\bigr)$. Using Lemma \ref{key31}
and the basis (\ref{Enyangbasis2}) of $\bb_n(-2m)$, we may assume
that $x=e_1e_3\cdots e_{2f-1}\sum_{d\in\mathfrak{D}_{f}}z_d d$,
where $\nu=\nu_f=((2^f),(n-2f))$ and the coefficients $z_d$, where
$d\in\mathfrak{D}_{f}$, are taken from
$K\BS_{\{2f+1,\cdots,n\}}\subseteq K\BS_{n}$. We then have to show
$x=0$.

Fix $\bk\in I_f$ and write $v=v_{\bc}\otimes v_{\bk}$. Let
$\lam^{(1)}, \lam^{(2)}$ be the $GL_m$-weights of $v_{\bc}$ and
$v_{\bk}$ respectively. Since $V^{\otimes n}$ is the direct sum of
its $GL_m$-weight spaces, we conclude $(vx)_{\mu}=0$ for all
$\mu\in\Lambda(m,n)$. In particular,
$$\begin{aligned} 0&=(vx)_{\lam}=\bigl((v_{\bc}\otimes
v_{\bk})x\bigr)_{\lam}=\sum_{d\in\mathfrak{D}_{f}}\Bigl(v_{\bc}e_1e_3\cdots
e_{2f-1}\otimes
v_{\bk}\Bigr)_{\lam}z_d d\\
&=\sum_{d\in\mathfrak{D}_{f}}\Bigl((v_{\bc}e_1e_3\cdots
e_{2f-1})_{\lam^{(1)}}\otimes v_{\bk}\Bigr)z_d d.\end{aligned}
$$
By definition, it is easy to see that
$$\bigl(v_{\bc}e_1e_3\cdots
e_{2f-1}\bigr)_{\lam^{(1)}}=\sum_{y\in\Psi}v_{{\bc}}y.
$$
Let us denote
this element by $\widehat{v}$.
Then $\sum_{d\in\mathfrak{D}_{f}}\bigl(\widehat{v}\otimes
v_{\bk}\bigr)z_d d=0$.\smallskip

We write $d=d_1d_{J}$, where $d_1\in\mathcal{D}_f,
J\in\mathcal{P}_f$. Then $$ \bigl(\widehat{v}\otimes
v_{\bk}\bigr)z_d d=\bigl(\widehat{v}\otimes v_{\bk}z_d\bigr)
d=\bigl(\widehat{v}\otimes v_{\bk}z_d\bigr)
d_1d_{J}=\bigl(\widehat{v}d_1\otimes v_{\bk}z_{d_1d_{J}}\bigr)d_{J}.
$$
If $J, L\in\mathcal{P}_f, J\neq L$, choose $1\leq l\leq n$ with
$l\in J$ but $l\not\in L$. Thus there exists an
$j\in\{1,2,\cdots,2f\}$ which is mapped by $d_J$ to $l$, but
$(l)d_L^{-1}>2f$. Note that for any $d\in\mathcal{D}_f$ all basis
vectors $v_i$ occurring in $\widehat{v}d$ as factors have index in
the set $\{1,2,\cdots,f\}$, and all those $v_i$ occurring in
$v_{\bk}z_{dd_J}$, respectively in $v_{\bk}z_{dd_L}$, have index $i$
between $2f+1$ and $m$. Let $v_{i_1}\otimes\cdots\otimes v_{i_n}$ be
a simple tensor involved in $\bigl(\widehat{v}d_1\otimes
v_{\bk}z_{d_1d_{J}}\bigr)d_{J}$ and $v_{j_1}\otimes\cdots\otimes
v_{j_n}$ be a simple tensor involved in $\bigl(\widehat{v}d_2\otimes
v_{\bk}z_{d_2d_L}\bigr)d_{L}$ for $d_1,d_2\in\mathcal{D}_f$. Then,
by the above, we have that $2f+1\leq j_l\leq m$, and $v_{i_l}=v_k$
for some $1\leq k\leq f$. Consequently the simple tensors $v_{\bi},
\bi\in I(m,n)$ involved in $\bigl\{(\widehat{v}d_1\otimes
v_{\bk}z_{d_1d_J})d_{J}\bigr\}$ and in
$\bigl\{(\widehat{v}d_2\otimes v_{\bk}z_{d_2d_L})d_{L}\bigr\}$ are
disjoint, hence both sets are linearly independent. We conclude that
$\sum_{d\in\mathcal{D}_f}\bigl(\widehat{v}d\otimes
v_{\bk}z_{dd_J}\bigr)d_{J}=0$ for each $J\in\mathcal{P}_f$, hence
$\sum_{d_1\in\mathcal{D}_f}\widehat{v}d_1\otimes
v_{\bk}z_{d_1d_J}=0$.

Note that $\widehat{v}d_1$ is a linear combination of basis tensors
$v_{\bi}=v_{i_1}\otimes\cdots\otimes v_{i_{2f}}$, with
$\bi\in\bc\Psi d_1$, and that we obtain by varying $d_1$ through
$\mathcal{D}_f$ precisely the partition of $\BS_{2f}$ into
$\Psi$-cosets. These are mutually disjoint. We conclude that the
basic tensors involved in $\widehat{v}d_1$ are disjoint for
different choices of $d_1\in\mathcal{D}_f$. Therefore, the equality
$\sum_{d_1\in\mathcal{D}_f}\widehat{v}d_1\otimes
v_{\bk}z_{d_1d_{J}}=0$ implies that $\widehat{v}d_1\otimes
v_{\bk}z_{d_1d_{J}}=0$ for each fixed $d_1\in\mathcal{D}_f$. Now we
vary $\bk\in I_f$. The $K$-span of $\bigl\{v_{\bk}\bigm|\bk\in
I_f\bigr\}$ is isomorphic to the tensor space $V^{\otimes n-2f}$ for
the symmetric group $\BS_{\{2f+1,\cdots,n\}}\cong\BS_{n-2f}$. Since
$m-2f\geq n-2f$, hence $\BS_{\{2f+1,\cdots,n\}}$ acts faithfully on
it. This implies $z_{d_1d_{J}}=0$ for all $d_1\in\mathcal{D}_{f},
J\in \mathcal{P}_f$. Thus $x=0$ and the lemma is proved.\end{proof}

The following corollary can be proved in exactly the same way as in
\cite[Corollary 3.11]{DDH}.

\begin{cor} Let $d\in\mathfrak{D}_{f}, \nu=\nu_f$. Then
$$B^{(f)}\cap\Bigl(\bigcap_{\bk\in I_f}\ann\bigl((v_{\bc}\otimes
v_{\bk})d\bigr) \Bigr)
= B^{(f+1)} \oplus \Biggl(\bigoplus  K\tilde{d}_{1}^{-1}e_1e_3\cdots
e_{2f-1}\sigma d_2\Biggr)
$$ where the rightmost direct sum is taken over all $\tilde{d}_1, d_2
\in \mathfrak{D}_{f}$ such that $\tilde{d}_1 \ne d$ and all $\sigma
\in \BS_{\{2f+1,\cdots,n\}}$. Hence
$B^{(f)}\cap\Bigl(\bigcap_{d\in\mathfrak{D}_{f}}\bigcap_{\bk\in
  I_f}\ann\bigl((v_{\bc}\otimes v_{\bk})d\bigr) \Bigr) =B^{(f+1)}$.
\end{cor}\bigskip

\noindent{\bf Proof of part b) in Theorem \ref{mainthm1} in the case
$m\geq n$}: We have seen that $\ann_{\bb_n(m)}\bigl(V^{\otimes
n}\bigr)\subseteq B^{(1)}$, and the above Corollary implies that
$\ann_{\bb_n(m)}\bigl(V^{\otimes n}\bigr)\subseteq B^{(f+1)}$
provided that $\ann_{\bb_n(m)}\bigl(V^{\otimes n}\bigr)\subseteq
B^{(f)}$. Thus by induction on $f$ we have
$\ann_{\bb_n(m)}\bigl(V^{\otimes n}\bigr)\subseteq B^{(f)}$ for all
natural numbers $f$. Since $B^{(f+1)}=0$ for $f>[n/2]$ it follows
that $\ann_{\bb_n(m)}\bigl(V^{\otimes n}\bigr)=0$. In other words,
$\widehat{\varphi}$ and hence $\varphi$ is injective if $m\geq n$.
By comparing dimension, we deduce that $\varphi$ is an isomorphism
onto
$$\End_{K GO(V)}\bigl(V^{\otimes n}\bigr)=\End_{K
O(V)}\bigl(V^{\otimes n}\bigr).$$ This completes the proof of part
b) in Theorem \ref{mainthm1} in the case $m\geq n$.

\bigskip\bigskip

\section{Proof of part b) in Theorem \ref{mainthm1} in the case $m<n$}

The purpose of this section is to give the proof the part b) in
Theorem \ref{mainthm1} in the case where $m<n$. Our approach is the
same as that used in \cite[Section 4]{DDH}.
\smallskip

To prove $\varphi(\bb_n(m))=\End_{KO(V)}\bigl(V^{\otimes n}\bigr)$,
we can assume without loss of
generality that $K=\overline{K}$ is algebraically closed. This is because, on the one hand, the $K$-dimension
of $\End_{KO(V)}\bigl(V^{\otimes n}\bigr)$ does not depend on the choice of the infinite field $K$; on the other hand,
the $\overline{K}$-dimension of $\varphi(\bb_n(m)_{\overline{K}})$ is the same as the $K$-dimension of $\varphi(\bb_n(m)_{K})$.
We fix $m_0\in\mathbb{N}$ such that $m_0\geq m$ and $m_0-m$ is even. We
denote by $\mathfrak{so}_{m_0}, \mathfrak{so}_{m}$ the special
orthogonal Lie algebras over $\mathbb{C}$. Let
$\widetilde{\mathfrak{g}}:=\mathfrak{so}_{m_0},
\mathfrak{g}:=\mathfrak{so}_{m}$. Recall that
$\mathcal{R}=\mathbb{Z}[1/2]$. Let $\bU_{\Q}$ (resp.,
$\bU_{\mathcal{R}}$) be the universal enveloping algebra of
$\mathfrak{g}$ over $\Q$ (resp., Kostant's $\mathcal{R}$-form in
$\bU_{\Q}$). Let $q$ be an indeterminant over $\mathcal{R}$. Let
$\mU_{\Q(q)}$ (resp., $\mU_{\mathcal{R}}$) be the Drinfel'd--Jimbo
quantized enveloping algebra of $\mathfrak{g}$ over $\Q(q)$ (resp.,
Lusztig's $\mathcal{R}[q,q^{-1}]$-form in $\mU_{\Q(q)}$). Let
$\bU_{K}:=\bU_{\mathcal{R}}\otimes_{\mathcal{R}}K,
\mU_{K}:=\mU_{\mathcal{R}}\otimes_{\mathcal{R}}K$. By putting a
``$\sim$" on the head, we can define similar notations for
$\widetilde{\mathfrak{g}}$.

\smallskip

Let $\widetilde{V}_{\mathcal{R}}$ be a free module of rank $m_0$ over
$\mathcal{R}$. Assume that $\widetilde{V}_{\mathcal{R}}$ is equipped
with a symmetric bilinear form $(\, ,\,  )$ as well as an ordered basis
$\bigl\{{v}_1,{v}_2,\cdots,{v}_{m_0}\bigr\}$ satisfying $({v}_i,
{v}_{j})=\delta_{i,m_0+1-j}$. For any commutative $\mathcal{R}$
algebra $K$, we set
$\widetilde{V}_{K}:=\widetilde{V}_{\mathcal{R}}\otimes_{\mathcal{R}}K$.
Let ${\iota}$ be the $K$-linear injection from $V_K \cong
V_{\mathcal{R}}\otimes_{\mathcal{R}}K$ into $\widetilde{V}_K$ defined
by
$$\sum_{i=1}^{m}k_iv_i\mapsto\sum_{i=1}^{m} k_i{v}_{i+(m_0-m)/2},\quad\forall\,k_1,\cdots,k_{m}\in
K.
$$
Let ${\pi}$ be the $K$-linear surjection from $\widetilde{V}_K$ onto
$V_K$ defined by
$$\sum_{i=1}^{m_0}k_i{v}_i\mapsto\sum_{i=1}^{m}k_{i+(m_0-m)/2}{v}_{i}, \quad\forall\,k_1,\cdots,k_{m_0}\in
K.$$ Then, ${\iota}$ induces an identification of ${\mathfrak{g}}$
as a subalgebra of $\widetilde{\mathfrak{g}}$, and also an
identification of $SO_m(K)$ (resp., $O_m(K)$) as a subgroup of
$SO_{m_0}(K)$ (resp., $O_{m_0}(K)$). Henceforth, we fix these
embeddings. The following result is well-known (cf.
\cite[Part I, Lemma 7.16, 7.17(6)]{Ja}).

\begin{lem} \label{lm51} We have
$$\begin{aligned}
\End_{\bU_{K}(\widetilde{\mathfrak{g}})}\bigl(\widetilde{V}_K^{\otimes
n}\bigr)&=\End_{KSO_{m_0}(K)}\bigl(\widetilde{V}_K^{\otimes
n}\bigr),\\
\End_{\bU_{K}(\mathfrak{g})}\bigl(V_K^{\otimes
n}\bigr)&=\End_{KSO_{m}(K)}\bigl({V}_K^{\otimes
n}\bigr).\end{aligned}
$$
\end{lem}

Note that the homomorphism $\iota$ and $\pi$ naturally induce a
linear map
$$\begin{aligned} \Theta_0: \End_{KSO_{m_0}(K)}\bigl(\widetilde{V}_K^{\otimes
n}\bigr)&\rightarrow\End_{KSO_m(K)}\bigl(V_K^{\otimes n}\bigr)\\
f&\mapsto \pi\circ f\circ\iota.
\end{aligned}$$ By restriction, we get a linear map (again denoted by
$\Theta_0$) from
$\End_{KO_{m_0}(K)}\bigl(\widetilde{V}_K^{\otimes n}\bigr)$ to
$\End_{KO_m(K)}\bigl(V_K^{\otimes n}\bigr)$. Note that, $\Theta_0$
is in general not an algebra map.\medskip

\begin{lem} \label{lm52} We have \begin{enumerate}
\item[(1)] the map $$
\Theta_0: \End_{KSO_{m_0}(K)}\bigl(\widetilde{V}_K^{\otimes
n}\bigr)\rightarrow\End_{KSO_m(K)}\bigl(V_K^{\otimes n}\bigr)
$$
is surjective.
\item[(2)] the map $$
\Theta_0: \End_{KO_{m_0}(K)}\bigl(\widetilde{V}_K^{\otimes
n}\bigr)\rightarrow\End_{KO_m(K)}\bigl(V_K^{\otimes n}\bigr)
$$
is surjective.
\end{enumerate}
\end{lem}

\begin{proof} We first prove (1). By Lemma \ref{lm51}, it suffices
to show that the map $$ \Theta_0:
\End_{\bU_{K}(\widetilde{\mathfrak{g}})}\bigl(\widetilde{V}_K^{\otimes
n}\bigr)\rightarrow\End_{\bU_{K}(\mathfrak{g})}\bigl(V_K^{\otimes
n}\bigr)
$$
is surjective. The same argument used in \cite[Section 4]{DDH} still
works (except that we use a slightly different embedding
$\mathfrak{g}\hookrightarrow\widetilde{\mathfrak{g}}$ here). So we
shall give only a sketch here. Recall that for a module $M$ over a
Hopf algebra $H$, $M^{H}:=\{x\in
M|hx=\varepsilon_{H}x,\,\forall\,h\in H\}$. We have the following
commutative diagram
$$\begin{CD}
\End_{\bU_{K}(\widetilde{\mathfrak{g}})}\bigl(\widetilde{V}_K^{\otimes
n}\bigr)@>{\sim}>>\bigl(\widetilde{V}_K^{\otimes
2n}\bigr)^{\bU_{K}(\widetilde{\mathfrak{g}})}\\
@V{\Theta_0}VV @V{\pi^{\otimes 2n}}VV  \\
\End_{\bU_{K}(\mathfrak{g})}\bigl(V_K^{\otimes
n}\bigr)@>{\sim}>>\bigl(V_K^{\otimes 2n}\bigr)^{\bU_{K}(\mathfrak{g})}\\
\end{CD},
$$
where the two horizontal maps are natural isomorphisms. Therefore,
it suffices to show that $$ {\pi^{\otimes
2n}}\Bigl(\bigl(\widetilde{V}_K^{\otimes
2n}\bigr)^{\bU_{K}(\widetilde{\mathfrak{g}})}\Bigr)=\bigl(V_K^{\otimes
2n}\bigr)^{\bU_{K}(\mathfrak{g})}.
$$
Since (by the theory of tilting modules) all the maps and modules
are defined over $\mathcal{R}$, it suffices to prove the above
equality with $K$ replaced by $\mathcal{R}$. Let $\widetilde{M}[\neq
0]_{\mathcal{R}}, M[\neq 0]_{\mathcal{R}}, B[0], \widetilde{B}[0]$
be the notations for $\mathfrak{g}, \widetilde{\mathfrak{g}}$ which
is defined in a similar way as in \cite[Section 4]{DDH}. We have the
following commutative diagram.
$$\begin{CD}
\bigl(\widetilde{V}_{\mathcal{R}}^{\otimes
2n}\bigr)^{\bU_{\mathcal{R}}(\widetilde{\mathfrak{g}})}@>{\sim}>>
\biggl(\widetilde{V}_{\mathcal{R}}^{\otimes 2n}/\widetilde{M}[\neq
0]_{\mathcal{R}}\biggr)^{\ast}@>{}>>
\biggl(\widetilde{V}_{\mathcal{R}}^{\otimes
2n}\biggr)^{\ast}\\
@V{{\pi}^{\otimes 2n}}VV @. @V{\bigl({\iota}^{\otimes 2n}\bigr)^{\ast}}VV  \\
\bigl(V_{\mathcal{R}}^{\otimes
2n}\bigr)^{\bU_{\mathcal{R}}(\mathfrak{g})}@>{\sim}>>\biggl(V_{\mathcal{R}}^{\otimes
2n}/M[\neq 0]_{\mathcal{R}}\biggr)^{\ast}
@>{}>>\biggl(V_{\mathcal{R}}^{\otimes 2n}\biggr)^{\ast}\\
\end{CD}.
$$
Hence it suffices to show that the rightmost vertical map is
surjective.\smallskip

Let $$\begin{aligned} J_0&:=\Bigl\{(i_1,\cdots,i_{2n})\in
I(m,2n)\Bigm|w_{i_1}\diamond\cdots\diamond w_{i_{2n}}\in
B[0]\Bigr\},\\
\widetilde{J}_0&:=\Bigl\{(i_1,\cdots,i_{2n})\in
I(m_0,2n)\Bigm|\tilde{w}_{i_1}\tilde{\diamond}
\cdots\tilde{\diamond}
\tilde{w}_{i_{2n}}\in\widetilde{B}[0]\Bigr\},\end{aligned}
$$
where $\{w_1,\cdots,w_m\}$ (resp.,
$\{\tilde{w}_1,\cdots,\tilde{w}_m\}$) are canonical bases of
$V_{\mathcal{R}}$ (resp., of $\widetilde{V}_{\mathcal{R}}$),
``$\diamond$" is the notation in \cite[(27.3.2)]{Lu}, and
``$\tilde\diamond$" is the similar notation for
$\widetilde{\mathfrak{g}}$. As in \cite[Corollary 4.5]{DDH}, it is
clear that the set $$ \Bigl\{w_{i_1}\otimes\cdots\otimes
w_{i_{2n}}+M[{\neq 0}]_{\mathcal{R}}\Bigm|(i_1,\cdots,i_{2n})\in
J_0\Bigr\}
$$
forms an $\mathcal{R}$-basis of $V_{\mathcal{R}}^{\otimes
{2n}}/M[{\neq 0}]_{\mathcal{R}}$, and the set $$
\Bigl\{\tilde{w}_{i_1}\otimes\cdots\otimes
\tilde{w}_{i_{2n}}+\widetilde{M}[{\neq
0}]_{\mathcal{R}}\Bigm|(i_1,\cdots,i_{2n})\in \widetilde{J}_0\Bigr\}
$$
forms an $\mathcal{R}$-basis of
$\widetilde{V}_{\mathcal{R}}^{\otimes {2n}}/\widetilde{M}[{\neq
0}]_{\mathcal{R}}$. We set $$
J_0[(m_0-m)/2]:=\Bigl\{((m_0-m)/2+i_1,\cdots,(m_0-m)/2+i_{2n})\Bigm|(i_1,\cdots,i_{2n})\in
J_0\Bigr\}.
$$

Using the same argument as in the proof of \cite[Theorem 4.7]{DDH},
we can show\footnote{Note that we get a shift here because we have
used a slight different embedding
$\mathfrak{g}\hookrightarrow\widetilde{\mathfrak{g}}$.} that
$J_0[(m_0-m)/2]\subseteq\widetilde{J}_0$. This implies that
${\iota}^{\otimes 2n}$ maps $V_{\mathcal{R}}^{\otimes 2n}/M[\neq
0]_{\mathcal{R}}$ isomorphically onto an $\mathcal{R}$-direct
summand of $\widetilde{V}_{\mathcal{R}}^{\otimes
2n}/\widetilde{M}[\neq 0]_{\mathcal{R}}$. It follows that $$
\bigl({\iota}^{\otimes
2n}\bigr)^{\ast}\biggl(\bigl(\widetilde{V}_{\mathcal{R}}^{\otimes
2n}/\widetilde{M}[\neq
0]_{\mathcal{R}}\bigr)^{\ast}\biggr)=\bigl(V_{\mathcal{R}}^{\otimes
2n}/M[\neq 0]_{\mathcal{R}}\bigr)^{\ast}, $$ as required. This
proves (1).\smallskip

For (2), we note that if $m$ is odd, then $m_0$ is odd too (as
$m_0-m$ is even). In this case, since $\theta$ acts as a scalar ($1$
or $-1$) on $V_K^{\otimes n}$, it is readily seen that
$$\begin{aligned} \End_{KO_{m_0}(K)}\bigl(\widetilde{V}_K^{\otimes
n}\bigr)&=\End_{KSO_{m_0}(K)}\bigl(\widetilde{V}_K^{\otimes
n}\bigr),\\
\End_{KO_m(K)}\bigl(V_K^{\otimes
n}\bigr)&=\End_{KSO_m(K)}\bigl(V_K^{\otimes n}\bigr).
\end{aligned}
$$
Therefore, by (1), we know $\Theta_0$ is surjective in this case.
Now we assume that $m$ is even, then $m_0$ is also even. We have the
following commutative diagram. $$\begin{CD}
\End_{KSO_{m_0}(K)}\bigl(\widetilde{V}_K^{\otimes
n}\bigr)@>{\sim}>>\bigl(\widetilde{V}_K^{\otimes
2n}\bigr)^{SO_{m_0}(K)}\\
@V{\Theta_0}VV @V{\pi^{\otimes 2n}}VV  \\
\End_{KSO_m(K)}\bigl(V_K^{\otimes
n}\bigr)@>{\sim}>>\bigl(V_K^{\otimes 2n}\bigr)^{SO_m(K)}\\
\end{CD},
$$
which implies (by (1)) that
\begin{equation}\label{hj} \pi^{\otimes
2n}\Bigl(\bigl(\widetilde{V}_K^{\otimes
2n}\bigr)^{SO_{m_0}(K)}\Bigr)=\bigl(V_K^{\otimes
2n}\bigr)^{SO_m(K)}.
\end{equation}
Since $\theta$ normalizes $SO_m(K)$, $\theta$ must stabilize
$\bigl(V_K^{\otimes 2n}\bigr)^{SO_m(K)}$. As $\theta^2=1$ and $1\neq
-1$, it follows that $\theta$ acts semisimply on $\bigl(V_K^{\otimes
2n}\bigr)^{SO_m(K)}$ with two eigenvalues $1, -1$, and
$\bigl(V_K^{\otimes 2n}\bigr)^{O_m(K)}$ is nothing but its
eigenspace belonging to $1$. By (\ref{hj}), $\pi^{\otimes 2n}$ must
map the eigenspace of $\bigl(\widetilde{V}_K^{\otimes
2n}\bigr)^{SO_{m_0}(K)}$ belonging to $1$ surjectively onto the
eigenspace of $\bigl({V_K}^{\otimes 2n}\bigr)^{SO_m(K)}$ belonging
to $1$. In other words, $$ \pi^{\otimes
2n}\Bigl(\bigl(\widetilde{V}_K^{\otimes
2n}\bigr)^{O_{m_0}(K)}\Bigr)=\bigl(V_K^{\otimes 2n}\bigr)^{O_m(K)}.
$$
Now the surjectivity of $\Theta_0$ in this case follows directly
from the following commutative diagram.
$$\begin{CD}
\End_{KO_{m_0}(K)}\bigl(\widetilde{V}_K^{\otimes
n}\bigr)@>{\sim}>>\bigl(\widetilde{V}_K^{\otimes
2n}\bigr)^{O_{m_0}(K)}\\
@V{\Theta_0}VV @V{\pi^{\otimes 2n}}VV  \\
\End_{KO_m(K)}\bigl(V_K^{\otimes
n}\bigr)@>{\sim}>>\bigl(V_K^{\otimes 2n}\bigr)^{O_m(K)}\\
\end{CD}.
$$
This completes the proof of the lemma.\end{proof}

Recall that the orthogonal bilinear form on $V$ determines an
$O_m(K)$-isomorphism $V_K\cong V_K^{\ast}$. Therefore, there is an
isomorphism
\begin{equation} \End_{K}\bigl(V_K^{\otimes n}\bigr)\cong V_K^{\otimes
n}\otimes (V_K^{\otimes n})^{\ast}\cong \bigl(V_K^{\otimes
2n}\bigr)^{\ast},\label{cano}
\end{equation}
such that for any given $\ui=(i_1,\cdots,i_n),
\underline{j}=(j_1,\cdots,j_n)\in I(m,n)$, the map which sends
$v_{\underline{l}}:=v_{l_1}\otimes\cdots\otimes v_{l_n}$ to
$\delta_{\ui,\underline{l}}v_{\uj}$ corresponds to the linear
function $$\begin{aligned} v_{k_1}\otimes v_{k_2}
\otimes\cdots\otimes v_{k_{2n}}&\mapsto
\delta_{m+1-j_1,k_{2n}}\delta_{m+1-j_2,k_{2n-1}}\cdots\delta_{m+1-j_n,k_{n+1}}\\
&\qquad\qquad\qquad\qquad
\times\delta_{k_1,i_1}\delta_{k_2,i_2}\cdots\delta_{k_n,i_n},
\end{aligned}$$ for any $(k_1,\cdots,k_{2n})\in I(m,2n)$. The
symmetric group $\BS_{2n}$ acts on $V_K^{\otimes 2n}$ by place
permutation, hence also acts on $\bigl(V_K^{\otimes
2n}\bigr)^{\ast}\cong\End_{K}\bigl(V_K^{\otimes n}\bigr)$. Similar
results also hold for $\widetilde{V}_K$ (with $m$ replaced by
$m_0$).
\smallskip

Let $D\in\BD_n$. We can write $D=D_{\ui,\uj}$, where
$$\ui=(i_1,\cdots,i_n), \uj=(j_1,\cdots,j_n),$$ such that
$(i_1,j_1,i_2,j_2,\cdots,i_n,j_n)$ is a permutation of $(1,2,$
$3,\cdots,2n)$, and for each integer $1\leq s\leq n$, the vertex
labelled by $i_s$ is connected with the vertex labelled by $j_s$.

\begin{lem} {\rm (cf. \cite{DP}, \cite[Proposition 1.6]{Ga1}, \cite{LP})} \label{keydes}
With the notations as above and (\ref{cano}) in mind, for any
$w_1,\cdots,w_{2n}\in V_K$, we have that
$$ \varphi(D_{\ui,\uj})\bigl(w_1\otimes\cdots\otimes
w_{2n}\bigr)=\prod_{s=1}^{n}(w_{i_s}, w_{j_s}).
$$
Furthermore, $\varphi$ is a $\BS_{2n}$-module homomorphism. Similar
results also hold for $\widetilde{V}_K$.
\end{lem}

\begin{proof} Let $d_1^{-1}e_1e_3\cdots e_{2f-1}\sigma
d_2$ be the basis element which corresponds to the Brauer diagram
$D$, where $f$ be an integer with $0\leq f\leq [n/2]$,
$\sigma\in\BS_{\{2f+1,\cdots,n\}}$ and $d_1,d_2\in\mathfrak{D}_{f}$,
$\nu_f:=\bigl((2^f), (n-2f)\bigr)\vdash n$. Then the top horizontal
edges of $D$ connect $(2i-1)d_1$ and $(2i)d_1$, the bottom
horizontal edges of $D$ connect $(2i-1)d_2$ and $(2i)d_2$, for
$i=1,2,\cdots,f$, and the vertical edges of $D$ connect $(j)d_1$ and
$(j)\sigma d_2$, for $j=2f+1,2f+2,\cdots,n$. \smallskip

Let $\ui\in I(m,n)$. By our definitions of $\varphi$ and the set
$\mathfrak{D}_{f}$, the action of
$$D=d_1^{-1}e_1e_3\cdots e_{2f-1}\sigma d_2$$ on $v_{\ui}$ can be
described as follows. Let $(a_1,b_1),\cdots,(a_f,b_f)$ be the set of
all the horizontal edges in the top row of $D$, where $a_s<b_s$ for
each $s$ and $a_1<a_2<\cdots<a_f$. Let $(c_1,d_1),\cdots,(c_f,d_f)$
be the set of all the horizontal edges in the bottom row of $D$,
where $c_s<d_s$ for each $s$ and $c_1<c_2<\cdots<c_f$. Then for each
integer $s$ with $1\leq s\leq f$, the $(c_s, d_s)$th position of
$v_{\ui}D$ is the following sum:
$$ \delta_{i_{a_s}, m+1-i_{b_s}}\sum_{k=1}^{m}(v_{k}\otimes v_{k'}).
$$
We list those vertices in the top row of $D$ which are not connected
with horizontal edges from left to right as
$i_{k_{2f+1}},i_{k_{2f+2}},\cdots,i_{k_{n}}$. Then, for each integer
$s$ with $2f+1\leq s\leq n$, the $(s\sigma d_2)$th position of
$v_{\ui}D$ is $v_{i_{k_{s}}}$. Now it is easy to verify directly
that $\varphi(D_{\ui,\uj})\bigl(w_1\otimes\cdots\otimes
w_{2n}\bigr)=\prod_{s=1}^{n}(w_{i_s}, w_{j_s})$, from which we see
immediately that $\varphi$ is a $\BS_{2n}$-module homomorphism.
\end{proof}

We define a linear isomorphism $\Theta_1$ from the $\bb_n(m_0)$ onto
$\bb_n(m)$ as follows: $$
\Theta_1\Bigl(\widetilde{d}_1^{-1}\widetilde{e}_1\widetilde{e}_3\cdots
\widetilde{e}_{2f-1}\widetilde{\sigma}\widetilde{d_2}\Bigr)={d_1}^{-1}e_1e_3\cdots
e_{2f-1}{\sigma}{d_2},
$$
for each $0\leq f\leq [n/2]$, $\lam\vdash n-2f$, $d_1,
d_2\in\mathfrak{D}_{f}$.

\begin{lem} \label{keycd} The following diagram of maps $$\begin{CD}
\bb_n(m_0)@>{\widetilde{\varphi}}>>\End_{KO_{m_0}(K)}\bigl(\widetilde{V}_K^{\otimes
n}\bigr)\\
@V{\Theta_1}VV @V{\Theta_0}VV  \\
\bb_n(m)@>{\varphi}>>\End_{KO_m(K)}\bigl(V_K^{\otimes
n}\bigr)\\
\end{CD}
$$
is commutative.
\end{lem}

\begin{proof} This follows directly from Lemma \ref{keydes}.
\end{proof}
\bigskip

\noindent{\bf Proof of Part b) in Theorem \ref{mainthm1} in the case
$m<n$}: Since $m_0\geq n$, by the main result in last section, we
know that $\widetilde{\varphi}$ is surjective. Since $\Theta_1$ is a
linear isomorphism, and by Lemma \ref{lm52} $\Theta_0$ is also
surjective, the commutativity of the diagram in Lemma \ref{keycd}
immediately implies that $\varphi$ is also surjective. This
completes the proof of Part b) in Theorem \ref{mainthm1} in the case
$m<n$.

\bigskip\bigskip

\section{The $\BS_{2n}$-action on $\bb_n(x)$}
\medskip

In this section, we shall first introduce (cf. \cite{FG}) the right
sign permutation action of the symmetric group $\BS_{2n}$ on the set
$\BD_n$. Then we shall construct a new $\Z$-basis for the resulting
right $\BS_{2n}$-module, which yields filtrations of $\bb_n(x)$ by
right $\BS_{2n}$-modules. Certain submodules occurring in this
filtration will play a central role in the next section.
\smallskip

For any fixed-point-free involution $\sigma$ in the symmetric group
$\BS_{2n}$, the conjugate $w^{-1}\sigma w$ of $\sigma$ by
$w\in\BS_{2n}$ is still a fixed-point-free involution. Therefore, we
have a right action of the symmetric group $\BS_{2n}$ on the set of
fixed-point-free involutions in $\BS_{2n}$. Note that the set
$\BD_n$ of Brauer $n$-diagrams can be naturally identified with the
set of fixed-point-free involutions in $\BS_{2n}$. Hence we get (cf.
\cite{FG}, \cite{Hu2}) a right permutation action of the symmetric
group $\BS_{2n}$ on the set $\BD_n$ of Brauer $n$-diagrams. We use
$``\ast"$ to denote this right permutation action. Let $\bb_{n}$
denote the free $\Z$-module spanned by all the Brauer $n$-diagrams
in $\BD_n$. The right sign permutation action of $\BS_{2n}$ on
$\bb_n$ is defined by $$D\star w:=(-1)^{\ell(w)}D\ast w.$$

We shall adopt the following labelling of the vertices in each
Brauer diagram. Namely, for each Brauer $n$-diagram $D$, we shall
label the vertices in the top row of $D$ by integers
$1,2,3,\cdots,n$ from left to right, and label the vertices in the
bottom row of $D$ by integers $n+1,n+2,n+3,\cdots,2n$ from right to
left (see Figure 6.1 for an example for $n=5$). This way of
labelling is more suitable for studying the sign permutation action
from $\BS_{2n}$.

\medskip
\begin{center}
\scalebox{0.25}[0.25]{\includegraphics{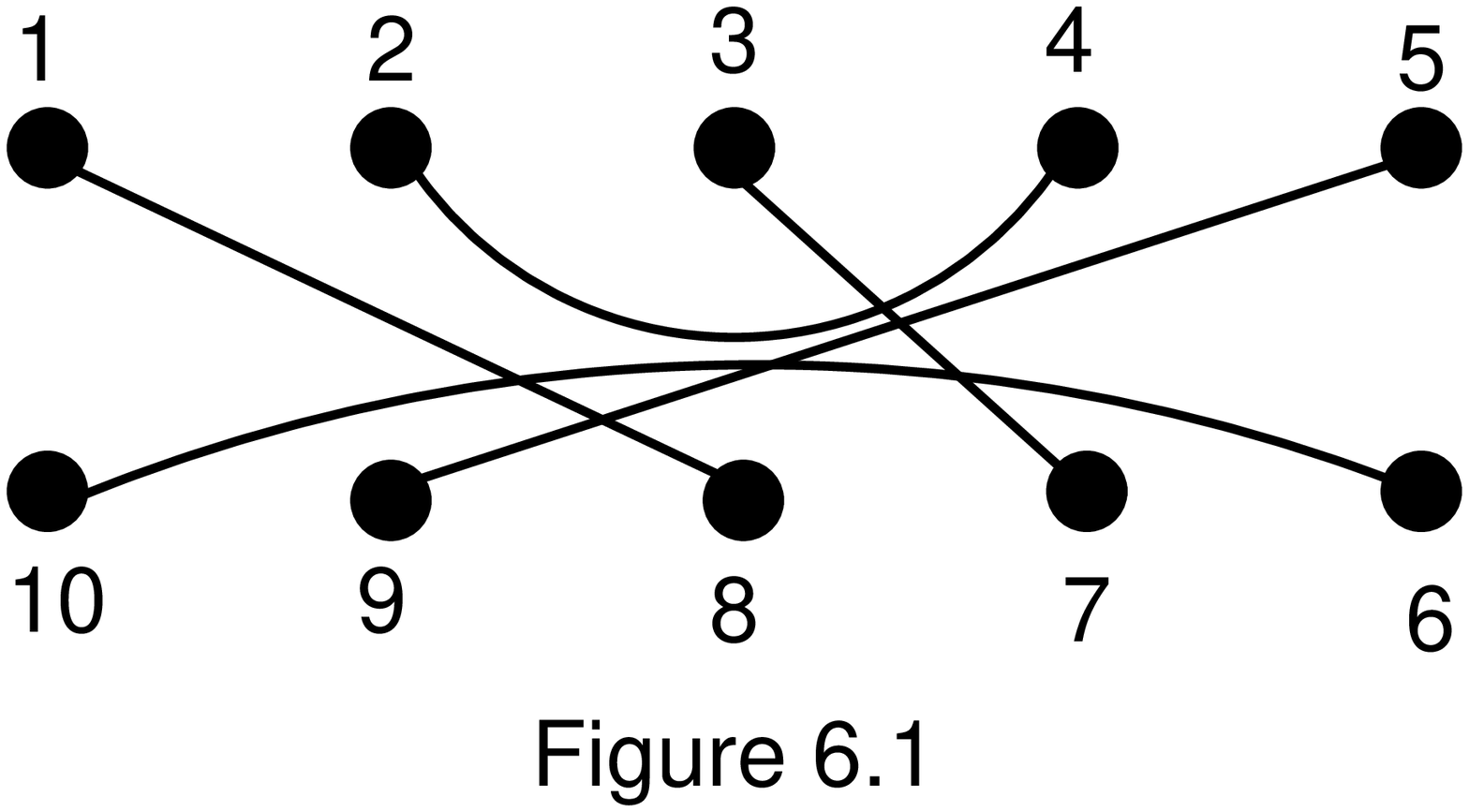}}
\end{center}
\medskip

For any commutative $\Z$-algebra $R$, we use $\bb_{n,R}$ to denote
the free $R$-module spanned by all the Brauer $n$-diagrams in
$\BD_n$. Then $\bb_{n,R}$ becomes a right $R[\BS_{2n}]$-module.
Clearly, there is a canonical isomorphism $\bb_{n,R}\cong
R\otimes_{\Z}\bb_n$, which is also a right $R[\BS_{2n}]$-module
isomorphism. Taking $R=\Z[x]$, we deduce that the Brauer algebra
$\bb_n(x)$ becomes a right $\Z[x][\BS_{2n}]$-module.\smallskip

For any $1\leq i\leq 2n$, we define $\gamma({i}):=2n+1-i$. Then
$\gamma$ is an involution on $\{1,2,\cdots,2n\}$. It is well-known
that the subgroup
$$ \bigl\{w\in\BS_{2n}\bigm|\text{$\bigl(\gamma(a)\bigr)w=\gamma(aw)$
  for any $1\leq a\leq 2n$}\bigr\}
$$
is isomorphic to the wreath product $\Z_2\!\wr\!\BS_n$ of $\Z_2$ and
$\BS_n$, which is a Weyl group of type $B_n$ (cf.
\cite{Hu1}).\smallskip

For any commutative $\Z$-algebra $R$, let $1_{R}$ denote the rank
one trivial representation of $R[\Z_2\wr\BS_n]$. By \cite[Lemma
2.1]{Hu2}, under the right permutation action $\ast$, there is a
right $R[\BS_{2n}]$-module isomorphism $$
\bb_{n,R}\cong\Ind_{R[\Z_2\wr\BS_n]}^{R[\BS_{2n}]}1_{R}.
$$
Let $\SGN_{R}$ be the rank one sign representation of $\BS_{2n}$.

\begin{lem}
Let $\BS_{2n}$ act on $\bb_{n,R}$ via the right sign permutation
action ``\,\,$\star$" . Then there is a right $R[\BS_{2n}]$-module
isomorphism $$
\bb_{n,R}\cong\SGN_{R}\otimes\Ind_{R[\Z_2\wr\BS_n]}^{R[\BS_{2n}]}1_{R}.
$$
\end{lem}

Let $k$ be a positive integer. A sequence of nonnegative integers
$\lam=(\lam_1,\lam_2,\cdots)$ is said to be a composition of $k$
(denoted by $\lam\vDash k$) if $\sum_{i\geq 1}\lam_i=k$. A composition
$\lam=(\lam_1,\lam_2,\cdots)$ is said to be a partition of $k$
(denoted by $\lam\vdash k$) if $\lam_1\geq\lam_2\geq\cdots$.  The
conjugate of $\lam$ is defined to be a partition
$\lam'=(\lam'_1,\lam'_2,\cdots)$, where $\lam'_j:=\#\{i|\lam_i\geq
j\}$ for $j=1,2,\cdots$. Denote by $\mathcal{P}_n$ the set of
partitions of $n$.  For any partition $\mu$ of $2n$, denote by
$S^{\mu}$ the associated Specht module over $\BS_{2n}$. The notion of
Specht modules we use here is the same as that introduced in
\cite{Mu}.  In particular, $S^{(2n)}$ is the one-dimensional trivial
representation of $\BS_{2n}$, while $S^{(1^{2n})}$ is the one
dimensional sign representation of $\BS_{2n}$. For any commutative $\Z$-algebra
$R$, we write $S_{R}^{\mu}:=R\otimes_{\Z}S^{\mu}$.  Then
$\bigl\{S_{\Q}^{\mu}\bigm|\mu\vdash 2n\bigr\}$ is a complete set of
pairwise non-isomorphic simple $\Q[\BS_{2n}]$-modules.
\smallskip

For any composition $\lam=(\lam_1,\cdots,\lam_s)$ of $n$, let
$2\lam:=(2\lam_1,\cdots,2\lam_s)$, which is a composition of $2n$.
We define $2\PP_n:=\bigl\{2\lam\bigm|\lam\in\PP_n\bigr\}$.

\begin{lem} \label{lm22}
Let $\BS_{2n}$ act on $\bb_{n,\Q}$ via the right sign permutation
action ``\,\,$\star$" . Then there is a right $\Q[\BS_{2n}]$-module
isomorphism $$ \bb_{n,\Q}\cong\bigoplus_{\lam\in
2\PP_n}S_{\Q}^{\lam'}.
$$
\end{lem}

\begin{proof}
This follows from \cite[Chapter VII, (2.4)]{Mac} (see \cite[Lemma 2.2]{Hu2}).
\end{proof}

For any non-negative integers $a,b$ with $a+b\leq n$, we denote by
$\BD_{(b)}^{(a)}$ the set of all the Brauer diagrams in $\BD_n$
which satisfy the following two conditions:

1) for any integer $i$ with $i\leq a$ or $i>a+b$, the vertex
labelled by $i$ is connected with the vertex labelled by
$\gamma(i)$;

2) for any integer $i$ with $a<i\leq a+b$, the vertex labelled by
$i$ is connected with the vertex labelled by $\gamma(j)$ for some
integer $j$ with $a<j\leq a+b$.

Note that one can naturally identify any Brauer diagram $D$ in
$\BD_{(b)}^{(a)}$ with an element $w(D)$ in
$\BS_{(a+1,a+2,\cdots,a+b)}$. Thus we can attach a sign
$\epsilon(D):=(-1)^{\ell(w(D))}$ to each Brauer diagram in
$\BD_{(b)}^{(a)}$, where $\ell(?)$ is the usual length function. We
define $$
 Y_{(b)}^{(a)}:=\sum_{D\in\BD_{(b)}^{(a)}}\epsilon(D)D
$$

For each integer $k$ with $0\leq k\leq n$, we set
$Y_{(k)}:=Y_{(k)}^{(0)}$.\smallskip

\begin{dfn} Let $\lam=(\lam_1,\cdots,\lam_{s})$ be a partition of $n$.
We define $$
Y_{\lam}:=Y_{(\lam_1)}^{(0)}Y_{(\lam_2)}^{(\lam_1)}\cdots
Y_{(\lam_{s})}^{(\lam_1+\lam_2+\cdots+\lam_{s-1})}\in\bb_n.
$$
\end{dfn}

We identify $\BS_n$ with the set of Brauer $n$-diagrams in $\bb_n$
which contain no horizontal edges. Let $\BS_{\lambda}$ be its
Young subgroup corresponding to $\lam$. Then we actually
have $Y_{\lam}\in\BS_{\lam}\subset\BS_n\subset\bb_n$. For any
positive integer $k$ and any composition
$\mu=(\mu_1,\cdots,\mu_s)$ of $k$, the Young diagram of $\mu$ is
defined to be the set $[\mu]:=\{(a,b)|1\leq a\leq s, 1\leq
b\leq\mu_a\}$. The elements of $[\mu]$ are called nodes of $\mu$.
A $\mu$-tableau $\ft$ is defined to be a bijective map from the
Young diagram $[\mu]$ onto the set $\{1,2,\cdots,k\}$. For each
integer $a$ with $1\leq a\leq k$, we define $\res_{\ft}(a)=j-i$ if
$\ft(i,j)=a$. We denote by $\ft^{\mu}$ the $\mu$-tableau in which
the numbers $1,2,\cdots,k$ appear in order along successive rows. We denote by $\ft_{\mu}$ the
$\mu$-tableau in which the numbers $1,2,\cdots,k$ appear in order
along successive columns. The row stabilizer of $\ft^{\mu}$, denoted by $\BS_{\mu}$, is the
Young subgroup of $\BS_k$ corresponding to $\mu$. We
define $$ x_{\mu}=\sum_{w\in\BS_{\mu}}w, \quad
y_{\mu}=\sum_{w\in\BS_{\mu}}(-1)^{\ell(w)}w.
$$
Let $w_{\mu}\in\BS_{k}$ be such that
$\ft^{\mu}w_{\mu}=\ft_{\mu}$.  For example, if
$k=8, \mu=(3,3,1,1)$, then $$ \ft^{\mu}=\begin{matrix} 1& 2& 3\\
4& 5& 6\\
7& &\\
8& &
\end{matrix},\,\,\,\,\,
\ft_{\mu}=\begin{matrix} 1& 5& 7\\
2& 6& 8\\
3& &\\
4& &
\end{matrix},\,\,\,\,\,w_{\mu}=(2,5,6,8,4)(3,7).
$$
For any partition $\lam=(\lam_1,\lam_2,\cdots,\lam_{s})$ of $n$,
we define
$$
\widehat{\lam}:=(\lam_1,\lam_2,\cdots,\lam_s,\lam_s,\cdots,\lam_2,\lam_1),
$$
which is a composition of $2n$. Let $\BS_{\widehat{\lam}}$ be the
Young subgroup of $\BS_{2n}$ corresponding to
$\widehat{\lam}$.

\begin{lem} \label{cor24}
Let $\lam=(\lam_1,\cdots,\lam_{s})$ be a partition of
$n$. Then for any $w\in\BS_{\widehat{\lam}}$, we have that
$$ Y_{\lam}\star w=Y_{\lam}.
$$
\end{lem}

\begin{proof}
Since the elements $Y_{(\lam_1)}^{(0)}, Y_{(\lam_2)}^{(\lam_1)}, \cdots,
Y_{(\lam_{s})}^{(\lam_1+\lam_2+\cdots+\lam_{s-1})}$ pairwise
commute with one another, the lemma follows directly from
definition.
\end{proof}

Let $k$ be a positive integer and $\mu$ be a composition of $k$.
Recall that a $\mu$-tableau $\ft$ is called {\it row standard} if
the numbers increase along rows. We use $\RS(\mu)$ to denote the
set of all the row-standard $\mu$-tableaux. Suppose $\mu$ is a
partition of $k$. Then $\ft$ is called {\it column standard} if
the numbers increase down columns, and {\it standard} if it is
both row and column standard. We use $\Std(\mu)$ to denote the set of
all the standard $\mu$-tableaux.

Note that every partition in the set $(2\PP_n)'$ is of the form
$$\widetilde{\nu}:=(\nu_1,\nu_1,\nu_2,\nu_2,\cdots,\nu_s,\nu_s),$$ where
$\nu:=(\nu_1,\nu_2,\cdots,\nu_s)$ is a partition of $n$. Now let
$\nu=(\nu_1,\cdots,\nu_{s})$ be a partition of $n$. For any
$\ft\in\RS(\widetilde{\nu})$, let $d(\ft)\in\BS_{2n}$ be such that
$\ft^{\widetilde{\nu}}d(\ft)=\ft$. Let $Y_{\nu,\ft}:=Y_{\nu}\star
d(\ft)$. For any commutative $\Z$-algebra $R$, we define $$
\CM^{\lam}_{R}:=\text{$R$-$\Span$}\Bigl\{Y_{\nu,\ft}\Bigm|\ft\in\Std(\widetilde{\nu}),\lam\unlhd\nu\in
\PP_n\Bigr\}.
$$
We write $\CM^{\lam}=\CM_{\Z}^{\lam}$. We are interested in the
module $\CM_{R}^{\lam}$. In the remaining part of this paper, we
shall see that this module is actually a right $\BS_{2n}$-submodule
of $\bb_{n,R}$ (with respect to the action ``$\star$"), and it
shares many properties with the permutation module
$x_{\widehat{\lam}}\Z[\BS_{2n}]$. In particular, it also has a
Specht filtration, and it is stable under base change, i.e.,
$R\otimes_{\Z}\CM^{\lam}\cong \CM_{R}^{\lam}$ for any commutative
$\Z$-algebra $R$.\smallskip

For our purpose, we need to recall some results in \cite{Mu} and
\cite{Ma} on the Specht filtrations of permutation modules over the
symmetric group $\BS_{2n}$. Let $\lam, \mu$ be two partitions of
$2n$. A $\mu$-tableau of type $\lam$ is a map
$\fS:[\mu]\rightarrow\{1,2,\cdots,2n\}$ such that each $i$ appears
exactly $\lam_i$ times. $\fS$ is said to be semistandard if each row
of $\fS$ is nondecreasing and each column of $\fS$ is strictly
increasing. Let $\T_0(\mu,\lam)$ be the set of all the semistandard
$\mu$-tableaux of type $\lam$. Then $\T_0(\mu,\lam)\neq\emptyset$
only if $\mu\unrhd \lam$, where ``$\unrhd$'' is the dominance order
as defined in \cite{Mu}. For each standard $\mu$-tableau $\fs$, let
$\lambda(\fs)$ be the tableau which is obtained from $\fs$ by
replacing each entry $i$ in $\fs$ by $r$ if $i$ appear in row $r$ of
$\ft^{\lam}$. Then $\lambda(\fs)$ is a $\mu$-tableau of type $\lam$.

For each standard $\mu$-tableau $\ft$ and each semistandard
$\mu$-tableau $\fS$ of type $\lam$, we define $$
x_{\fS,\ft}:=\sum_{\fs\in\Std(\mu),\lambda(\fs)=\fS}d(\fs)^{-1}x_{\mu}d(\ft).
$$
Then by \cite[Section 7]{Mu}, the set $$
\Bigl\{x_{\fS,\ft}\Bigm|\fS\in\T_0(\mu,\lam), \ft\in\Std(\mu), \lam\unlhd\mu\vdash 2n\Bigr\}
$$
form a $\Z$-basis of $x_{\lam}\Z[\BS_{2n}]$. Furthermore, for any
commutative $\Z$-algebra $R$, the canonical surjective homomorphism
$R\otimes_{\Z}x_{\lam}\Z[\BS_{2n}]\twoheadrightarrow
x_{\lam}R[\BS_{2n}]$ is an isomorphism.

For each partition $\mu$ of $2n$ and for each semistandard
$\mu$-tableau $\fS$ of type $\lam$, according to the results in
\cite[Section 7]{Mu} and \cite{Ma}, both the following
$\Z$-submodules $$\begin{aligned}
M^{\lam}_{\fS}:&=\text{$\Z$-$\Span$}\Bigl(\bigl\{x_{\fS,\fs}\bigm|\fs\in\Std(\mu)\bigr\}\bigcup\\
&\qquad\qquad\qquad
\bigl\{x_{\fT,\ft}\bigm|\fT\in\T_0(\nu,\lam),\ft\in\Std(\nu),\mu\lhd\nu\vdash 2n\bigr\}\Bigr),\\
M^{\lam}_{\fS,\rhd}:&=\text{$\Z$-$\Span$}\Bigl\{x_{\fT,\ft}\Bigm|\fT\in\T_0(\nu,\lam),
\ft\in\Std(\nu),\mu\lhd\nu\vdash 2n\Bigr\},
\end{aligned}
$$
are $\Z[\BS_{2n}]$-submodules, and the quotient of $M^{\lam}_{\fS}$
by $M^{\lam}_{\fS,\rhd}$ is canonical isomorphic to $S^{\mu}$ so
that the images of the elements $x_{\fS,\fs}$, where
$\fs\in\Std(\mu)$, form the standard $\Z$-basis of $S^{\mu}$. In
this way, it gives rise to a Specht filtration of
$x_{\lam}\Z[\BS_{2n}]$.  Each semistandard $\mu$-tableau of type
$\lam$ yields a factor which is isomorphic to $S^{\mu}$ so that
$x_{\lam}\Z[\BS_{2n}]$ has a series of factors, ordered by $\unlhd$,
each isomorphic to some $S^{\mu}$, $\mu\unrhd \lam$; moreover, the
multiplicity of $S^{\mu}$ is the number of semistandard
$\mu$-tableaux of type $\lam$.

Let $\lam=(\lam_1,\cdots,\lam_{s})$ be a partition of $n$, where
$\lam_s>0$. We write
$$\widehat{\lam}=(a_1^{k_1},a_2^{k_2},\cdots,a_{s'}^{k_{s'}},a_{s'}^{k_{s'}},\cdots,a_2^{k_2},a_1^{k_1}),$$
where $a_1>a_2>\cdots>a_{s'}$, $k_i\in\mathbb{N}$ for each $i$,
$a_i^{k_i}$ means that $a_i$ repeats $k_i$ times. Let
$\widehat{\lam}'$ be the conjugate of $\widehat{\lam}$. Let
$\widetilde{\BS}_{\widehat{\lam}}$ be the subgroup of
$\BS_{\widehat{\lam}'}$ consisting of all the elements $w$
satisfying the following condition: for any integers $1\leq i, j\leq
s$ with $\widehat{\lam}_i=\widehat{\lam}_j$, and any integers $a,b$
with $1\leq a, b\leq\widehat{\lam}_i$, $$
\begin{aligned}
&\qquad\,\,(\ft_{\widehat{\lam}}(i,a))w=\ft_{\widehat{\lam}}(j,a)\\
&\Leftrightarrow\quad
\ft_{\widehat{\lam}}(i,b))w=\ft_{\widehat{\lam}}(j,b)\\
&\Leftrightarrow\quad
\ft_{\widehat{\lam}}(2s+1-i,a))w=\ft_{\widehat{\lam}}(2s+1-j,a)\\
&\Leftrightarrow\quad\ft_{\widehat{\lam}}(2s+1-i,b))w=\ft_{\widehat{\lam}}(2s+1-j,b).
\end{aligned}
$$

Let $\widetilde{D}_{\widehat{\lam}}$ be a complete set of right
coset representatives of $\widetilde{\BS}_{\widehat{\lam}}$ in
$\BS_{\widehat{\lam}'}$.

\begin{lem} \label{lm26} Let $\lam=(\lam_1,\cdots,\lam_{s})$ be a partition of $n$, where $\lam_s>0$. We keep
the notations as above. Let
$$ n_{\widehat{\lam}}:=\prod_{i=1}^{s'}2^{k_i}(k_i!),\quad
h_{\widehat{\lam}}:=\sum_{w\in\widetilde{D}_{\widehat{\lam}}}(-1)^{\ell(w)}w.
$$
Then $$ Y_{\lam}\star
\bigl(w_{\widehat{\lam}}y_{\widehat{\lam}'}\bigr)=n_{\widehat{\lam}}\bigl(Y_{\lam}\star
(w_{\widehat{\lam}}h_{\widehat{\lam}})\bigr),
$$
and for any commutative $\Z$-algebra $R$,
$1_{R}\otimes_{\Z}(Y_{\lam}\star
(w_{\widehat{\lam}}h_{\widehat{\lam}}))\neq 0$ in $\bb_{n,R}$.
\end{lem}

\begin{proof} By definition, $$
y_{\widehat{\lam}'}=\sum_{w\in\BS_{\widehat{\lam}'}}(-1)^{\ell(w)}w=
\Bigl(\sum_{w\in\widetilde{\BS}_{\widehat{\lam}}}(-1)^{\ell(w)}w\Bigr)h_{\widehat{\lam}}.
$$
By definition, it is easy to see that for any
$w\in\widetilde{\BS}_{\widehat{\lam}}$, $\ell(w)$ is an even
integer. Now the first statement of this lemma follows from the
following identity: $$ (Y_{\lam}\star
w_{\widehat{\lam}})\star\Bigl(\sum_{w\in\widetilde{\BS}_{\widehat{\lam}'}}w\Bigr)=
n_{\widehat{\lam}}\bigl(Y_{\lam}\star w_{\widehat{\lam}}\bigr).
$$
Let $d$ be the Brauer $n$-diagram in which the vertex labelled by
$\ft_{\widehat{\lam}}(i,r)$ is connected with the vertex labelled by
$\ft_{\widehat{\lam}}(2s+1-i,r)$ for any $1\leq i\leq s, 1\leq
r\leq\lam_i$. Then it is easy to see that $d$ appears with
coefficient $(-1)^{\ell(w_{\widehat{\lam}})}$ in the expression of
$Y_{\lam}\star (w_{\widehat{\lam}}h_{\widehat{\lam}})$ as linear
combinations of the basis of Brauer $n$-diagrams. It follows that
for any commutative $\Z$-algebra $R$,
$1_{R}\otimes_{\Z}(Y_{\lam}\star
(w_{\widehat{\lam}}h_{\widehat{\lam}}))\neq 0$ in $\bb_{n,R}$, as
required.
\end{proof}

Following \cite{Mu0}, we define the Jucys-Murphy operators of
$\Z[\BS_{2n}]$.
$$\left\{\begin{aligned}
L_1:&=0,\\
L_a:&=(a-1,a)+(a-2,a)+\cdots+(1,a),\quad a=2,3,\cdots,2n.
\end{aligned}
\right.$$ Let $\lam=(\lam_1,\cdots,\lam_{s})$ be a partition of
$n$. Then $\widetilde{\lam}$ is the unique partition obtained by
reordering the parts of $\widehat{\lam}$. Let
$z_{\widetilde{\lam}}:=x_{\widetilde{\lam}}w_{\widetilde{\lam}}y_{\widetilde{\lam}'}$.
By \cite[Lemma 4.3]{DJ1}, there is a $\Z[\BS_{2n}]$-module
isomorphism from $S^{\widehat{\lam}}$ onto $S^{\widetilde{\lam}}$,
which maps $z_{\widehat{\lam}}$ to $\pm z_{\widetilde{\lam}}$. Now
applying \cite[(3.14)]{DJ2}, for each integer $1\leq a\leq 2n$, we
deduce that
$$
\bigl(x_{\widehat{\lam}}w_{\widehat{\lam}}y_{\widehat{\lam}'}\bigr)L_a=
\res_{\ft_{\widetilde{\lam}}}(a)\bigl(x_{\widehat{\lam}}w_{\widehat{\lam}}y_{\widehat{\lam}'}\bigr).
$$
For each standard $\widetilde{\lam}$-tableau $\ft$, we define $$
\Theta_{\ft}:=\prod_{i=1}^{n}\prod_{\substack{\fu\in\Std(\widetilde{\lam})\\
\res_{\fu}(i)\neq\res_{\ft}(i)}}\frac{L_i-\res_{\fu}(i)}{\res_{\ft}(i)-\res_{\fu}(i)}.
$$
By Lemma \ref{cor24} and Frobenius reciprocity, there is a
surjective  right $\Z[\BS_{2n}]$-module homomorphism $\pi_{\lam}$
from $x_{\widehat{\lam}}\Z[\BS_{2n}]$ onto $Y_{\lam}\Z[\BS_{2n}]$
which extends the map $x_{\widehat{\lam}}\mapsto Y_{\lam}$. In
particular, by Lemma \ref{lm26}, $$ \bigl(Y_{\lam}\star
w_{\widehat{\lam}}h_{\widehat{\lam}}\bigr)\star
L_a=\res_{\ft_{\widetilde{\lam}}}(a)\bigl(Y_{\lam}\star
(w_{\widehat{\lam}}h_{\widehat{\lam}})\bigr).
$$

\begin{prop} \label{prop27}
Let $\lam=(\lam_1,\cdots,\lam_{s})$ be a partition of $n$. We have
that
$$[Y_{\lam}\Q[\BS_{2n}]:S_{\Q}^{\widetilde{\lam}}]=1. $$
\end{prop}

\begin{proof} By Lemma \ref{lm22}, we have that $$
\bb_{n,\Q}\cong\bigoplus_{\mu\in (2\PP_n)'}S_{\Q}^{\mu}.
$$
It is well-known that each $S_{\Q}^{\mu}$ has a basis
$\bigl\{v_{\ft}\bigr\}_{\ft\in\Std(\mu)}$ satisfying $$
v_{\ft}L_i=\res_{\ft}(i)v_{\ft},\quad \forall\,1\leq i\leq n.
$$
Since $Y_{\lam}\Q[\BS_{2n}]\subseteq \bb_{n,\Q}$, we can write $$
Y_{\lam}\star (w_{\widehat{\lam}}h_{\widehat{\lam}})=\sum_{\mu\in
(2\PP_n)'}\sum_{\ft\in\Std(\mu)}A_{\ft}v_{\ft},
$$
where $A_{\ft}\in\Q$ for each $\ft$.

For each $\mu\in (2\PP_n)'$ and each $\ft\in\Std(\mu)$, we apply the
operator $\Theta_{\ft}$ on both sides of the above identity and use
Lemma \ref{lm26} and the above discussion. We get that $A_{\ft}\neq
0$ if and only if $\mu=\widetilde{\lam}$ and
$\ft=\ft_{\widetilde{\lam}}$. In other words, $Y_{\lam}\star
(w_{\widehat{\lam}}h_{\widehat{\lam}})=A_{\ft_{\widetilde{\lam}}}v_{\ft_{\widetilde{\lam}}}$
for some $0\neq A_{\ft_{\widetilde{\lam}}}\in\Q$. This implies that
the projection from $Y_{\lam}\Q[\BS_{2n}]$ to
$S_{\Q}^{\widetilde{\lam}}$ is nonzero. Hence,
$$[Y_{\lam}\Q[\BS_{2n}]:S_{\Q}^{\widetilde{\lam}}]=1,$$ as required.
\end{proof}

Suppose $\lam\in\PP_n$. Let
$\mathcal{D}_{\widehat\lam,\widetilde{\lam}}$ be the set of
distinguished $\BS_{\widehat{\lam}}$-$\BS_{\widetilde{\lam}}$
double coset representatives in $\BS_{2n}$ (cf. \cite{DJ1}). By
\cite[(1.1)]{DJ3},
$d_{\lam}^{-1}\BS_{\widehat{\lam}}d_{\lam}=\BS_{\widetilde{\lam}}$
for some $d_{\lam}\in\mathcal{D}_{\widehat\lam,\widetilde{\lam}}$.
Hence $x_{\widehat{\lam}}d_{\lam}=d_{\lam}x_{\widetilde{\lam}}$.
Then it is easy to see that the set $$
\Bigl\{d_{\lam}x_{\fS,\ft}\Bigm|\fS\in\T_0(\mu,\widetilde{\lam}),
\ft\in\Std(\mu), \widetilde{\lam}\unlhd\mu\vdash 2n\Bigr\}
$$
forms a $\Z$-basis of $x_{\widehat{\lam}}\Z[\BS_{2n}]$, and the
sets $d_{\lam}M_{\fS}^{\widetilde{\lam}}, d_{\lam}M_{\fS, \unrhd
}^{\widetilde{\lam}}$ define Specht filtrations for
$x_{\widehat{\lam}}\Z[\BS_{2n}]$.

By the natural surjective right $\Z[\BS_{2n}]$-module homomorphism
$\pi_{\lam}$ from $x_{\widehat{\lam}}\Z[\BS_{2n}]$ onto
$Y_{\lam}\Z[\BS_{2n}]$, we know that the elements
$\pi_{\lam}\bigl(d_{\lam}x_{\fS,\ft}\bigr)$, where
$\fS\in\T_0(\mu,\widetilde{\lam}), \ft\in\Std(\mu),
\widetilde{\lam}\unlhd\mu\vdash 2n$, span $Y_{\lam}\Z[\BS_{2n}]$
as $\Z$-module. Recall our definition of $\CM^{\lam}$ in the
paragraph below Lemma \ref{cor24}.

\begin{prop} \label{prop28}
Let $\lam$ be as in the previous proposition. For any partition $\mu$
of $2n$ and any $\fS\in\T_0(\mu,\widetilde{\lam})$, we have that
$\pi_{\lam}\bigl(d_{\lam}M_{\fS}^{\widetilde{\lam}}\bigr)\subseteq\CM^{\lam}$.
In particular, $Y_{\lam}\Z[\BS_{2n}]\subseteq \CM^{\lam}$.
\end{prop}

\begin{proof} We first prove a weak version of the claim in this
proposition. That is, for any partition $\mu$ of $2n$ and any
$\fS\in\T_0(\mu,\widetilde{\lam})$, $$
\pi_{\lam}\bigl(d_{\lam}M_{\fS}^{\widetilde{\lam}}\bigr)\subseteq\CM_{\Q}^{\lam}.$$

We consider the dominance order $``\unlhd"$ and make induction on
$\lam$. We start with the partition $\lam=(n)$, which is the
unique maximal partition of $n$ with respect to $``\unlhd"$. Then
$\widehat{\lam}=(n,n)=\widetilde{\lam}$ and $d_{\lam}=1$. Let
$$
S:=\begin{pmatrix}1,1,\cdots,1\\
2,2,\cdots,2\end{pmatrix}
$$
be the unique semistandard $(n,n)$-tableau in $\T_0((n,n),(n,n))$.
Since $\Q\otimes_{\Z}M_{\fS, \rhd}^{\widetilde{\lam}}$ contains no
composition factors in
$\bigl\{S_{\Q}^{\widetilde{\nu}}\bigm|\nu\in(2\PP_n)'\bigr\}$, it
follows that
$\pi_{\lam}\bigl(M_{\fS,\rhd}^{\widetilde{\lam}}\bigr)=0$. Hence
$\pi_{\lam}$ induces a surjective homomorphism $$
S_{\Q}^{(n,n)}\cong
M_{\fS}^{\widetilde{\lam}}/M_{\fS,\rhd}^{\widetilde{\lam}}\twoheadrightarrow
Y_{\lam}\Q[\BS_{2n}],
$$
by which it is easy to see the claim in this proposition is true
for $\lam=(n)$.

Now let $\lam\lhd (n)$ be a partition of $n$. Assume that for any
partition $\nu$ of $n$ satisfying $\nu\rhd\lam$, the claim in this
proposition is true. We now prove the claim for the partition
$\lam$.

Let $\mu\rhd\lam$ be a partition of $2n$ with
$\T_0(\mu,\widetilde{\lam})\neq\emptyset$. We consider again the
dominance order $``\unlhd"$ and make induction on $\mu$. Since
$\T_0((2n),\widetilde{\lam})$ contains a unique element
$\fS_{\star}$, $\Std((2n))=\{\ft^{(2n)}\}$, by Lemma \ref{lm22},
it is clear that
$$
\pi_{\lam}\bigl(d_{\lam}x_{\fS_{\star},\ft^{(2n)}}\bigr)=\pi_{\lam}(x_{(2n)})
=0\in\CM^{\lam}.$$ So in this case the claim of this proposition
is still true.

Now let $\mu\rhd\widetilde{\lam}$ be a partition of $2n$ with
$\T_0(\mu,\widetilde{\lam})\neq\emptyset$ and $\mu\lhd (2n)$.
Assume that for any partition $\nu$ of $2n$ satisfying
$\T_0(\nu,\widetilde{\lam})\neq\emptyset$ and $\nu\rhd\mu$, $$
\pi_{\lam}\bigl(d_{\lam}M_{\fS}^{\widetilde{\lam}}\bigr)\subseteq\CM_{\Q}^{\lam},
$$
for any $S\in\T_0(\nu,\widetilde{\lam})$.

Let $\fS\in\T_0(\mu,\widetilde{\lam})$. The homomorphism
$\pi_{\lam}$ induces a surjective map from
$d_{\lam}M_{\fS}^{\widetilde{\lam}}/d_{\lam}M_{\fS,\rhd}^{\widetilde{\lam}}$
onto
$$
\Bigl(\pi_{\lam}(d_{\lam}M_{\fS}^{\widetilde{\lam}})\Bigr)/\Bigl(\pi_{\lam}(d_{\lam}M_{\fS,\rhd}^{\widetilde{\lam}})\Bigr).
$$
Hence it also induces a surjective map $\widetilde{\pi}_{\lam}$
from $$
\Bigl(\Q\otimes_{\Z}d_{\lam}M_{\fS}^{\widetilde{\lam}}/\Bigl(\Q\otimes_{\Z}d_{\lam}M_{\fS,\rhd}^{\widetilde{\lam}}\Bigr)
\cong
\Q\otimes_{\Z}\Bigl(d_{\lam}M_{\fS}^{\widetilde{\lam}}/d_{\lam}M_{\fS,\rhd}^{\widetilde{\lam}}\Bigr)\cong
S_{\Q}^{\mu}
$$
onto $$
\Q\otimes_{\Z}\Bigl(\pi_{\lam}(d_{\lam}M_{\fS}^{\widetilde{\lam}})/\pi_{\lam}(d_{\lam}M_{\fS,\rhd}^{\widetilde{\lam}})\Bigr).
$$
Since $S_{\Q}^{\mu}$ is irreducible, the above map is either a
zero map or an isomorphism. If it is a zero map, then (by
induction hypothesis) $$
\pi_{\lam}(d_{\lam}M_{\fS}^{\widetilde{\lam}})\subseteq\pi_{\lam}(d_{\lam}M_{\fS,\rhd}^{\widetilde{\lam}})\subseteq
\CM_{\Q}^{\lam}.
$$
It remains to consider the case where $\widetilde{\pi}_{\lam}$ is
an isomorphism. In particular,
$$\Q\otimes_{\Z}\Bigl(\pi_{\lam}(d_{\lam}M_{\fS}^{\widetilde{\lam}})/\pi_{\lam}(d_{\lam}M_{\fS,\rhd}^{\widetilde{\lam}})\Bigr)\cong
S_{\Q}^{\mu}.
$$
Applying Lemma \ref{lm22}, we know that $\mu\in (2\PP_n)'$.
Therefore we can write $\mu=\widetilde{\nu}$ for some
$\nu\in\PP_n$. Note that $\mu\rhd\widetilde{\lam}$ implies that
$\nu\rhd\lam$.

On the other hand, there is also a surjective homomorphism
$\pi_{\nu}$ from
$x_{\widehat{\nu}}\Z[\BS_{2n}]/d_{\nu}M_{\fS_0,\rhd}^{\widetilde{\nu}}$
onto
$$
\Bigl(\pi_{\nu}(x_{\widehat{\nu}}\Z[\BS_{2n}])\Bigr)/\Bigl(\pi_{\nu}(d_{\nu}M_{\fS_0,\rhd}^{\widetilde{\nu}})\Bigr)=
Y_{\nu}\Z[\BS_{2n}]/\Bigl(\pi_{\nu}(d_{\nu}M_{\fS_0,\rhd}^{\widetilde{\nu}})\Bigr),
$$
where $\fS_0$ is the unique semistandard $\mu$-tableau in
$\T_0(\mu,\widetilde{\nu})$ in which the numbers
$1,1,\cdots,1,2,2,\cdots,2,\cdots,2n,\cdots,2n$ appears in order
along successive rows. Hence it also induces a surjective map
$\widetilde{\pi}_{\nu}$ from $$
\Bigl(\Q\otimes_{\Z}x_{\widehat{\nu}}\Z[\BS_{2n}]/\Bigl(\Q\otimes_{\Z}d_{\nu}M_{\fS_0,\rhd}^{\widetilde{\nu}}\Bigr)\cong
\Q\otimes_{\Z}\Bigl(x_{\widehat{\nu}}\Z[\BS_{2n}]/d_{\nu}M_{\fS_0,\rhd}^{\widetilde{\nu}}\Bigr)\cong
S_{\Q}^{{\mu}}
$$
onto $$
\Q\otimes_{\Z}\Bigl(Y_{\nu}\Z[\BS_{2n}]/\pi_{\nu}(d_{\nu}M_{\fS_0,\rhd}^{\widetilde{\nu}})\Bigr)\cong
\bigl(\Q\otimes_{\Z}Y_{\nu}\Z[\BS_{2n}]\bigr)/\bigl(\Q\otimes_{\Z}\pi_{\nu}(d_{\nu}M_{\fS_0,\rhd}^{\widetilde{\nu}})\bigr).
$$

It is well-known that $S_{\Q}^{\mu}$ does not occur as composition
factor in $\Q\otimes_{\Z}d_{\nu}M_{\fS_0,\rhd}^{\widetilde{\nu}}$.
Hence $S_{\Q}^{\mu}$ does not occur as composition factor in
$$\Q\otimes_{\Z}\pi_{\nu}(d_{\nu}M_{\fS_0,\rhd}^{\widetilde{\nu}}).$$
By Proposition \ref{prop27}, $S_{\Q}^{\mu}$ occurs as composition
factor with multiplicity one in
$\Q\otimes_{\Z}Y_{\nu}\Z[\BS_{2n}]$. Therefore,
$$\Q\otimes_{\Z}Y_{\nu}\Z[\BS_{2n}]\neq\Q\otimes_{\Z}\pi_{\nu}(d_{\nu}M_{\fS_0,\rhd}^{\widetilde{\nu}}).$$
It follows that $\widetilde{\pi}_{\nu}$ must be an isomorphism.
Hence
$$
\Q\otimes_{\Z}\Bigl(Y_{\nu}\Z[\BS_{2n}]/\pi_{\nu}(d_{\nu}M_{\fS_0,\rhd}^{\widetilde{\nu}})\Bigr)\cong
S_{\Q}^{\mu}.
$$

We write $A=\pi_{\lam}(d_{\lam}M_{\fS}^{\widetilde{\lam}}),
B=Y_{\nu}\Z[\BS_{2n}]$. Since $S_{\Q}^{\mu}$ appears only once in
$\bb_{n,\Q}$, it follows that $S_{\Q}^{\mu}$ must occur as
composition factor in the module
$$ \bigl(\Q\otimes_{\Z}A\bigr)\cap
\bigl(\Q\otimes_{\Z}B\bigr)=\Q\otimes_{\Z}(A\cap B).
$$
Hence $S_{\Q}^{\mu}$ can not occur as composition factor in the module $$
\bigl(\Q\otimes_{\Z}A\bigr)/\Bigl(\Q\otimes_{\Z}(A\cap B)\Bigr)\cong\Q\otimes_{\Z}(A/A\cap B).
$$
Therefore, the image of the canonical projection
$\Q\otimes_{\Z}A\rightarrow\Q\otimes_{\Z}(A/A\cap B)$ must  be
contained in the image of
$\Q\otimes_{\Z}\pi_{\lam}(d_{\lam}M_{\fS,\rhd}^{\widetilde{\lam}})$.
However, by induction hypothesis, both
$\pi_{\lam}(d_{\lam}M_{\fS,\rhd}^{\widetilde{\lam}})$ and $B$ are
contained in the $\Q$-span of
$\Bigl\{Y_{\alpha,\fu}\Bigm|\fu\in\Std(\widetilde{\alpha}),
\lam\unlhd\alpha\in \PP_n\Bigr\}$. It follows that $$
\pi_{\lam}\bigl(d_{\lam}M_{\fS}^{\widetilde{\lam}}\bigr)\subseteq\CM_{\Q}^{\lam},
$$
as required.

Suppose that $$
\pi_{\lam}\bigl(d_{\lam}M_{\fS}^{\widetilde{\lam}}\bigr)\not\subseteq\CM^{\lam}.
$$
Then (by the $\Z$-freeness of $\bb_n$) there exists an element
$x\in d_{\lam}M_{\fS}^{\widetilde{\lam}}$, integers $a, a_{\fu}$,
and a prime divisor $p\in\mathbb{N}$ of $a$, such that
$$ a\pi_{\lam}(x)=\sum_{\lam\unlhd\alpha\in
\PP_n}\sum_{\fu\in\Std(\widetilde{\alpha})}a_{\fu}Y_{\alpha}\star
d(\fu),
$$
and $\Sigma_p:=\bigl\{\alpha\in
\PP_n\bigm|\text{$\lam\unlhd\alpha, p\nmid a_{\fu}$, for some
$\fu\in\Std(\widetilde{\alpha})$}\bigr\}\neq\emptyset$.

We take an $\alpha\in\Sigma_p$ such that $\alpha$ is minimal with
respect to ``$\unlhd$". Then we take an
$u\in\Std(\widetilde{\alpha})$ such that $p\nmid a_{\fu}$ and
$\ell(d(\fu))$ is maximal among the elements in the set
$\bigl\{\fu\in\Std(\widetilde{\alpha})\bigm|p\nmid
a_{\fu}\bigr\}$. Let $\sigma_{\fu}$ be the unique element in
$\BS_{2n}$ such that $d(\fu)\sigma_{\fu}=w_{\alpha}$ and
$\ell(w_{\alpha})=\ell(d(\fu))+\ell(\sigma_{\fu})$. We consider
the finite field $\mathbb{F}_p$ as a $\Z$-algebra. By
\cite[(4.1)]{DJ1}, we know that for any composition $\beta$ of
$2n$, any $\gamma\in\PP_{2n}$, and element $w\in\BS_{2n}$, $$
\text{$x_{\beta}wy_{\gamma'}\neq 0$ only if
$\gamma\unrhd\beta$};\,\,\text{while}\,\,\text{$x_{\beta}wy_{\beta'}\neq
0$ only if $w\in\BS_{\beta}w_{\beta}$}.
$$
Hence for any $\beta\in\PP_n, \gamma\in\PP_{2n}$, $$
\text{$Y_{\beta}\star (wy_{\gamma'})\neq 0$ only if
$\gamma\unrhd\widehat{\beta}$};\,\,\text{$Y_{\beta}\star
(wy_{\widehat{\beta}'})\neq 0$ only if
$w\in\BS_{\widehat{\beta}}w_{\widehat{\beta}}$}.
$$ Now applying Lemma \ref{lm26}, we get $$
0=1_{\mathbb{F}_p}\otimes_{\Z}\bigl(a\pi_{\lam}(x)\star
(\sigma_{\fu}h_{\widehat{\alpha}})\bigr)=\pm
1_{\mathbb{F}_p}\otimes_{\Z}\bigl(a_{\fu}Y_{\alpha}\star
(w_{\widehat{\alpha}}h_{\widehat{\alpha}})\bigr)\neq 0,
$$
which is a contradiction. This proves that
$\pi_{\lam}\bigl(d_{\lam}M_{\fS}^{\widetilde{\lam}}\bigr)\subseteq\CM^{\lam}$.
\end{proof}

\begin{cor} \label{cor29}
For any partition $\lam\in \PP_n$ and any commutative $\Z$-algebra $R$,
$\CM_{R}^{\lam}$ is a right $\BS_{2n}$-submodule of $\bb_{n,R}$. \end{cor}

\begin{proof} This follows directly from Proposition \ref{prop28}.
\end{proof}

\begin{thm} \label{thm210}
For any partition $\lam\in \PP_n$ and any commutative $\Z$-algebra
$R$, the canonical map
$R\otimes_{\Z}\CM^{\lam}\rightarrow\CM_{R}^{\lam}$ is an isomorphism,
and the set $$
\Bigl\{Y_{\nu,\ft}\Bigm|\ft\in\Std(\widetilde{\nu}),\lam\unlhd\nu\in
\PP_n\Bigr\}
$$
forms an $R$-basis of $\CM_{R}^{\lam}$. In particular, the set $$
\Bigl\{Y_{\lam,\ft}\Bigm|\ft\in\Std(\widetilde{\lam}),\lam\in
\PP_n\Bigr\}
$$
forms an $R$-basis of $\bb_{n,R}$.
\end{thm}

\begin{proof} We take $\lam=(1^n)$, then
  $Y_{\lam}\Z[\BS_{2n}]=\bb_n$.
It is well-known that $
\bb_{n,R}\cong R\otimes_{\Z}\bb_{n}$ for any commutative
$\Z$-algebra $R$. Applying Proposition \ref{prop28}, we get that for
any commutative $\Z$-algebra $R$, the set
$$ \Bigl\{Y_{\lam,\ft}\Bigm|\ft\in\Std(\widetilde{\lam}),\lam\in
\PP_n\Bigr\}
$$
must form an $R$-basis of $\bb_{n,R}$. By the $R$-linear
independence of the elements in this set and Corollary
\ref{cor29}, we also get that, for any partition $\lam\in 2\PP_n$,
the set $$
\Bigl\{Y_{\nu,\ft}\Bigm|\ft\in\Std(\widetilde{\nu}),\lam\unlhd\nu\in
\PP_n\Bigr\}
$$
must form an $R$-basis of $\CM_{R}^{\lam}$. Therefore, for any
commutative $\Z$-algebra $R$, the canonical map
$R\otimes_{\Z}\CM^{\lam}\rightarrow \CM_{R}^{\lam}$ is an
isomorphism.
\end{proof}

\begin{thm} \label{thm211}
For any partition $\lam\in \PP_n$ and any commutative $\Z$-algebra
$R$, we define $$
\CM_R^{\rhd\lam}:=\text{$R$-$\Span$}\Bigl\{Y_{\nu,\ft}\Bigm|\ft\in\Std(\widetilde{\nu}),\lam\lhd\nu\in
\PP_n\Bigr\}.
$$
Then $\CM_R^{\rhd\lam}$ is a right $R[\BS_{2n}]$-submodule of
$\CM_R^{\lam}$, and there is a $R[\BS_{2n}]$-module isomorphism $$
\CM_R^{\lam}/\CM_R^{\rhd\lam}\cong S_R^{\widetilde{\lam}}.
$$
In particular, $\bb_{n,R}$ has a Specht filtration.
\end{thm}

\begin{proof}
It suffices to consider the case where $R=\Z$. We first show that $$
\CM_{\Q}^{\lam}\cong\oplus_{\lam\unlhd\mu\in
\PP_n}S_{\Q}^{\widetilde{\mu}},\quad
\CM_{\Q}^{\rhd\lam}\cong\oplus_{\lam\lhd\mu\in
\PP_n}S_{\Q}^{\widetilde{\mu}}.
$$
For each $\mu\in \PP_n$, we use $\rho_{\mu}^{\lam}$ to denote the
composite of the embedding
$\CM_{\Q}^{\lam}\hookrightarrow\bb_{n,\Q}$ and the projection
$\bb_{n,\Q}\twoheadrightarrow S_{\Q}^{\widetilde{\mu}}$. Suppose
that $\rho_{\mu}^{\lam}\neq 0$. Then $\rho_{\mu}^{\lam}$ must be a
surjection. We claim that $\mu\unrhd\lam$. In fact, if
$\mu\ntrianglerighteq\lam$, then for any $\lam\unlhd\nu\in \PP_n$,
$\mu\ntrianglerighteq\nu$, and
$x_{\widehat{\nu}}\Z[\BS_{2n}]w_{\widehat{\mu}'}x_{\widehat{\mu}}w_{\widehat{\mu}}y_{\widehat{\mu}'}=0$,
hence $Y_{\nu,\ft}\star
(w_{\widehat{\mu}'}x_{\widehat{\mu}}w_{\widehat{\mu}}y_{\widehat{\mu}'})=0$
for any $\ft\in\Std(\widetilde{\nu})$. It follows that
$\CM_{\Q}^{\lam}(w_{\widehat{\mu}'}x_{\widehat{\mu}}w_{\widehat{\mu}}y_{\widehat{\mu}'})=0$.
Therefore,
$S_{\Q}^{\widetilde{\mu}}(w_{\mu'}x_{\widehat{\mu}}w_{\widehat{\mu}}y_{\widehat{\mu}'})=0$.
On the other hand, since $S_{\Q}^{\widetilde{\mu}}\cong
x_{\widehat{\mu}}w_{\widehat{\mu}}y_{\widehat{\mu}'}\Q[\BS_{2n}]$,
and by \cite[Lemma 5.7]{KM},
$$
x_{\widehat{\mu}}w_{\widehat{\mu}}y_{\widehat{\mu}'}(w_{\widehat{\mu}'}x_{\widehat{\mu}}
w_{\widehat{\mu}}y_{\widehat{\mu}'})=\Bigl(\prod_{(i,j)\in[\widehat{\mu}]}h_{i,j}^{\widehat{\mu}}\Bigr)
x_{\widehat{\mu}}w_{\widehat{\mu}}y_{\widehat{\mu}'}\neq 0,
$$
where $h_{i,j}^{\widehat{\mu}}$ is the $(i,j)$-hook length in
$[\widehat{\mu}]$, we get a contradiction. Therefore,
$\rho_{\mu}^{\lam}\neq 0$ must imply that $\mu\unrhd\lam$. Now
counting the dimensions, we deduce that
$\CM_{\Q}^{\lam}\cong\oplus_{\lam\unlhd\mu\in
\PP_n}S_{\Q}^{\widetilde{\mu}}$. In a similar way, we can prove
that $\CM_{\Q}^{\rhd\lam}\cong\oplus_{\lam\lhd\mu\in
\PP_n}S_{\Q}^{\widetilde{\mu}}$. It follows that
$\CM_{\Q}^{\lam}/\CM_{\Q}^{\rhd\lam}\cong
S_{\Q}^{\widetilde{\lam}}$.

We now consider the natural map from
$x_{\widehat{\lam}}\Z[\BS_{2n}]$ onto $\CM^{\lam}/\CM^{\rhd\lam}$.
Since $\Q\otimes_{\Z}d_{\lam}M_{\fS_0,\rhd}^{\widetilde{\lam}}$
does not contain $S_{\Q}^{\widetilde{\lam}}$ as a composition
factor, it follows that (by Proposition \ref{prop28}) the image of
$d_{\lam}M_{\fS_0,\rhd}^{\widetilde{\lam}}$ must be $0$. Therefore
we get a surjective map from $S^{\widetilde{\lam}}$ onto
$\CM^{\lam}/\CM^{\rhd\lam}$. This map sends the standard basis of
$S^{\widetilde{\lam}}$ to the canonical basis of
$\CM^{\lam}/\CM^{\rhd\lam}$. So it must be injective as well, as
required.
\end{proof}
\bigskip\bigskip

\section{The second main result}
\medskip

In this section, we shall use Theorem \ref{mainthm1} and the results
obtained in Section 6 to give an explicit and characteristic-free
description of the annihilator of the $n$-tensor space $V^{\otimes
n}$ in the Brauer algebra $\bb_n(m)$.
\medskip

Let $K$ be an arbitrary infinite field of odd characteristic. Let
$m, n\in\mathbb{N}$. Let $V$ be the $m$-dimensional
orthogonal $K$-vector space we introduced before. Let
$O({V})$ be the corresponding orthogonal group, acting
naturally on $V$, and hence on the $n$-tensor space $V^{\otimes n}$
from the left-hand side. As we mentioned in the introduction, this
left action on $V^{\otimes n}$ is centralized by the specialized
Brauer algebra $\bb_n(m)_K:=K\otimes_{\Z}\bb_n(m)$, where $K$ is
regarded as $\Z$-algebra in a natural way. The Brauer algebra
$\bb_n(m)_K$ acts on $n$-tensor space $V^{\otimes n}$ from the
right-hand side. Let $\varphi$ be the natural $K$-algebra
homomorphism $$
\varphi:(\bb_n(m)_K)^{\mathrm{op}}\rightarrow\End_{K}\bigl(V^{\otimes
n}\bigr).
$$
Recall (see (\ref{cano})) that there is an isomorphism
$\End_{K}\bigl(V^{\otimes n}\bigr)\cong \bigl(V^{\otimes
2n}\bigr)^{\ast}$, and the place permutation action of the
symmetric group $\BS_{2n}$ on $V^{\otimes 2n}$ naturally induces an
action on $\End_{K}\bigl(V^{\otimes n}\bigr)\cong\bigl(V^{\otimes
2n}\bigr)^{\ast}$. By Lemma \ref{keydes}, we know that $\varphi$ is
a $\BS_{2n}$-module homomorphism. Using the
$O_m(\mathbb{C})$-$\mathbb{C}\BS_{2n}$-bimodule decomposition of
$V_{\mathbb{C}}^{\otimes 2n}$ (cf. \cite{KT}), it is easy to check
that $$
\dim\Bigl(V_{\mathbb{C}}^{\otimes 2n}\Bigr)^{O_m(\mathbb{C})}=\sum_{\substack{\lam\in(2\PP_n)'\\
\lam_1\leq m}}\dim S^{\lam}.
$$
Now applying Theorem \ref{mainthm1}, Lemma \ref{key2} and our
previous discussion, we deduce that

\begin{lem} \label{lm32} With the notations as
above, we have that $$ \dim(\Ker\varphi)=\sum_{\substack{\lam\in(2\PP_n)'\\
\lam_1>m}}\dim S^{\lam}.
$$
\end{lem}

We remark that when $K=\mathbb{C}$, the above result was deduced in
the work of \cite{LP} and \cite[Proposition 1.6]{Ga1}.
\bigskip\bigskip

\noindent {\bf Proof of Theorem \ref{mainthm2}:} For any
$\lam\in\PP_n$, it is easy to see that $\lam\unrhd (m+1,1^{n-m-1})$
if and only if $\lam_1>m$. Therefore, by Theorem \ref{thm210} and
Lemma \ref{lm32}, $\dim\Ker\varphi=\dim\CM_{K}^{(m+1,1^{n-m-1})}$.
Furthermore, by Lemma \ref{lm32}, $\Ker\varphi$ is a
$\BS_{2n}$-submodule of $\bb_n(m)$. Therefore, by Lemma \ref{lm32},
to prove the theorem, it suffices to show that
$Y_{\lam}\in\Ker\varphi$ for each partition $\lam\in\PP_n$
satisfying $\lam_1>m$.

By definition,
$Y_{\lam}:=Y_{(\lam_1)}^{(0)}Y_{(\lam_2)}^{(\lam_1)}\cdots
Y_{(\lam_{s})}^{(\lam_1+\lam_2+\cdots+\lam_{s-1})}$. By \cite{Ha},
we know that $Y_{(\lam_1)}^{(0)}\in\Ker\varphi$ whenever $\lam_1>m$.
It follows that $Y_{\lam}\in\Ker\varphi$ as required. This completes
the proof of the theorem.

\bigskip\bigskip\bigskip\bigskip

\bibliographystyle{amsplain}

\begin{thebibliography}{99}

\bibitem{AR} {\bibname A.M. Adamovich \and G.L. Rybnikov}, 
`Tilting modules for classical groups and Howe duality in positive
charactersitic', {\em Transformation groups} {1} (1996) 1--33.

\bibitem{AM} {\bibname M.F. Atiyah \and I.G. Macdonald}, 
{\em Introduction to commutative algebra}
(Addison--Wesley, 1969).

\bibitem{B} {\bibname R. Brauer}, `On algebras which are connected with
semisimple continuous groups', {\em Ann. of Math.} {38} (1937)
857--872.

\bibitem{B1} {\bibname W. P. Brown}, `An algebra related to the orthogonal
group', {\em Michigan Math. J.} {3} (1955--1956) 1--22.

\bibitem{B2} {\bibname W. P. Brown}, `The semisimplicity of $\omega_f^n$',
{\em Ann. of Math.} {63} (1956) 324--335.

\bibitem{C} {\bibname G. Cliff}, 
`A basis of bideterminants for the coordinate ring of the orthogonal
group', {\em Communications in Algebra} (7) 36 (2008) 2719--2749.

\bibitem{CDM} {\bibname A. Cox, M. De Visscher \and P. Martin}, 
`The blocks of the Brauer algebra in characteristic zero', Preprint,
2006, arXiv:math.RT/0601387.

\bibitem{CL} {\bibname R. W. Carter \and G. Lusztig}, `On the modular
representations of general linear and symmetric groups', {\em Math.
Z.} {136} (1974) 193--242.

\bibitem{DG} {\bibname M. Demazure \and P. Gabriel}, 
{\em Introduction to algebraic geometry and algebraic groups}
(North--Holland, Amsterdam, 1970).

\bibitem{DDH} {\bibname R. Dipper, S. Doty \and J. Hu}, 
`Brauer algebras, symplectic Schur algebras and Schur-Weyl duality',
{\em Trans. Amer. Math. Soc.} {360} (2008) 189--213.

\bibitem{DJ1} {\bibname R. Dipper \and G. D. James}, 
`Representations of Hecke algebras of general linear groups', {\em
Proc. London. Math. Soc.} (3) {52} (1986) 20--52.

\bibitem{DJ2} {\bibname R. Dipper \and G. D. James}, 
`Blocks and idempotents of Hecke algebras of general linear groups',
{\em Proc.  London. Math. Soc.} (3) {54} (1987) 57--82.

\bibitem{DJ3} {\bibname R. Dipper \and G. D. James}, 
`$q$-tensor space and $q$-Weyl modules', {\em Trans. Amer. Math. Soc.}
{327} (1991) 251--282.

\bibitem{Do} {\bibname S. Donkin}, {\em Rational representations of
algebraic groups}, Lect. Notes in Math. Vol. 1140 (Springer-Verlag,
1985).

\bibitem{Do1} {\bibname S. Donkin}, 
`On Schur algebras and related algebras I', {J. Algebra} { 104} (1986)
310--328.

\bibitem{Do2} {\bibname S. Donkin}, 
`On Schur algebras and related algebras II', {J. Algebra} { 111}
(1987) 354--364.

\bibitem{DP} {\bibname C. De Concini \and C. Procesi}, `A characteristic free
approach to invariant theory', {\em Adv. Math.} {21} (1976)
330--354.

\bibitem{Dt} {\bibname S. Doty}, `Polynomial representations, algebraic
monoids, and Schur algebras of classical type', {\em J. Pure Appl.
Algebra} {123} (1998) 165--199.

\bibitem{Dt2} {\bibname S. Doty}, 
`Representation theory of reductive normal algebraic monoids', {\em
Trans. Amer. math. Soc.} {351} (1999) 2539--2551.

\bibitem{Ei} {\bibname D. Eisenbud}, 
{\em Commutative algebra, with a view toward algebraic geometry},
Graduate Texts in Mathematics, {150} (Springer-Verlag, 1994).

\bibitem{E} {\bibname J. Enyang}, `Cellular bases for the Brauer and
Birman-Murakami-Wenzl algebras', {\em J. Algebra} {281} (2004)
413--449.

\bibitem{FG} {\bibname S. Fishel \and I. Grojnowski}, 
`Canonical bases for the Brauer centralizer algebra', {\em
Math. Res. Lett.} (1) {2} (1995) 15--26.

\bibitem{Ga1} {\bibname F. Gavarini}, 
`A Brauer algebra theoretic proof of Littlewood's restriction rules',
{\em J. Algebra} {212} (1999) 240--271.

\bibitem{Ga2} {\bibname F. Gavarini}, `On the radical of Brauer
algebras', {\em Math. Zeit.} 260 (2008) 673--697.

\bibitem{GL} {\bibname J. J. Graham \and G. I. Lehrer}, `Cellular algebras',
{\em Invent. Math.} {123} (1996) 1--34.

\bibitem{GW} {\bibname R. Goodman \and N. R. Wallach}, 
{\em Representations and invariants of classical groups} (Cambridge
University Press, 1998).

\bibitem{Gr} {\bibname J. A. Green}, {\em Polynomial representations of
$GL_{n}$}, Lect. Notes in Math. Vol. 830 (Springer-Verlag, 1980).

\bibitem{Ha} {\bibname M. H\"arterich}, 
`Murphy bases of generalized Temperley-Lieb algebras', {\em
Arch. Math.} (5) {72} (1999) 337--345.

\bibitem{Hu1} {\bibname J. Hu}, 
`Quasi-parabolic subgroups of the Weyl group of type $D$', {\em
European Journal of Combinatorics}, (3) {28} (2007) 807--821.

\bibitem{Hu2} {\bibname J. Hu}, 
`Specht filtrations and tensor spaces for the Brauer algebra', {\em
J. Algebr. Comb.}, (2) 28 (2008), 281--312.

\bibitem{HW1} {\bibname P. Hanlon \and D. B. Wales}, `On the decomposition
of Brauer's centralizer algebras', {\em J. Algebra} {121} (1989)
409--445.

\bibitem{HW2} {\bibname P. Hanlon \and D. B. Wales}, `Eigenvalues connected
with Brauer's centralizer algebras', {\em J. Algebra} {121} (1989)
446--475.

\bibitem{Ja} {\bibname J. C. Jantzen}, {\em Representations of Algebraic
Groups}, second edition (American Mathematical Society, 2003).

\bibitem{LP} {\bibname J.-L. Loday \and C. Procesi}, `Homology of symplectic
and orthogonal algebras', {\em Adv. Math.} {21} (1988) 93--108.

\bibitem{Lu} {\bibname G. Lusztig}, {\em Introduction to Quantum
Groups}, Progress in Math., {110} (Birkh\"auser, Boston, 1990).

\bibitem{KM} {\bibname M. K\"unzer \and A. Mathas}, 
`Elementary divisors of Specht modules', {\em European
J. Combinatorics}, {26} (2005) 943--964.

\bibitem{KSX} {\bibname S. K\"onig, I.H. Slung{\aa}rd \and C. Xi}, 
`Double centralizer properties, dominant dimension, and tilting
modules', {\em J. Algebra} {240} (2001) 393--412.

\bibitem{KT} {\bibname K. Koike \and I. Terada}, 
`Young-diagrammatic methods for the representation theory of the
classical groups of type $B_n$, $C_n$, $D_n$', {\em J. Algebra} { 107}
(1987) 466--511.

\bibitem{Liu} {\bibname  Q. Liu}, 
`Schur algebras of classical groups', {\em J. Algebra} 301 (2006),
867--887.

\bibitem{Liu2} {\bibname Q. Liu}, 
`Schur algebras of classical groups II', preprint 2007, University of
Koeln.

\bibitem{Mac} {\bibname I.G. Macdonald}, 
{\em Symmetric functions and Hall polynomials},
Second Edition, Oxford Mathematical Monographs (The Clarendon Press,
Oxford University Press, New York, 1995).

\bibitem{Ma} {\bibname A. Mathas}, 
{\em Iwahori-Hecke algebras and Schur algebras of the symmetric
group}, University lecture series, {15} (American Mathematical
Society, Providence, R.I., 1999).

\bibitem{Mu0} {\bibname E. Murphy}, 
`On the representation theory of the symmetric groups and associated
Hecke algebras', {\em J. Algebra} {152} (1992) 492--513.

\bibitem{Mu} {\bibname E. Murphy}, `The representations of Hecke algebras
of type $A_n$', {\em J. Algebra} {173} (1995) 97--121.

\bibitem{Oe1} {\bibname S. Oehms}, `Centralizer coalgebras,
FRT-construction, and symplectic monoids', {\em J. Algebra} (1)
{244} (2001) 19--44.

\bibitem{Sc} {\bibname I. Schur}, `\"Uber die rationalen Darstellungen der
allgemeinen linearen Gruppe', (1927). Reprinted in {\em I. Schur,
Gesammelte Abhandlungen}, Vol. III, pp. 68--85 (Springer-Verlag,
Berlin, 1973).

\bibitem{W} {\bibname H. Weyl}, {\em The classical groups, their invariants
and representations} (Princeton University Press, 1946).

\bibitem{Wz} {\bibname H. Wenzl}, 
`On the structure of Brauer's centralizer algebras', {\em Ann. of
   Math.} {128} (1988) 173--193.

\end{thebibliography}

\affiliationone{Stephen Doty\\
Department of Mathematics and Statistics\\
Loyola University Chicago\\
6525 North Sheridan Road\\
Chicago IL 60626\\
USA \email{doty@math.luc.edu}}

\affiliationone{Jun Hu\\
Department of Applied Mathematics\\
Beijing Institute of Technology\\
Beijing 100081\\
P.R. China\\
\email{junhu303@yahoo.com.cn}}
\end{document}